\documentclass[journal]{IEEEtran}
\ifCLASSINFOpdf
\else
\fi
\usepackage{epsfig}\usepackage{epsfig}
\usepackage{amssymb}
\usepackage{amsthm}
\usepackage{amsmath}
\usepackage{cite}
\usepackage{mathrsfs}
\usepackage{cases}
\usepackage{algorithm}
\usepackage{multirow}
\usepackage{makecell}
\usepackage{algorithmic}
\usepackage{subeqnarray}
\usepackage{subfigure}
\usepackage{pifont,bbm,amsfonts,mathrsfs,amsmath,stmaryrd,amssymb}
\usepackage{amscd,graphicx,array,dsfont,texdraw,tikz,footnote,verbatim}%
\usepackage[normalem]{ulem}
\usepackage{mathtools}
\usepackage{bbm,bbding,amssymb,pifont,mathrsfs,amsfonts,graphicx,mathtools,color}%
\usepackage{amscd,array,enumerate,dsfont,texdraw,tikz,multicol,bbm}

\usetikzlibrary{arrows,shapes,snakes,shadows,positioning,automata,patterns}
\usetikzlibrary{trees,decorations.pathmorphing,decorations.markings}
\usepackage{setspace} 
\usepackage{tikz}
\usepackage{pgf}
\usepackage{amsmath, amsfonts, amssymb, amsthm}
\usepackage{pgfplots}
\pgfplotsset{compat=newest}
\usepackage{xxcolor}
\usetikzlibrary{arrows, decorations.markings}
\usetikzlibrary{arrows,shadows,petri}

\newcommand{\beq}{\begin{equation}}
\newcommand{\eeq}{\end{equation}}
\newcommand{\bqa}{\begin{eqnarray}}
\newcommand{\eqa}{\end{eqnarray}}

\definecolor{maroon}{rgb}{0.7,0,0}

\definecolor{ngreen}{rgb}{0.3,0.7,0.3}

\definecolor{golden}{rgb}{0.8,0.6,0.1}

\DeclareMathOperator*{\argmin}{arg\,min}
\newtheorem{prop}{\indent Proposition}
\newtheorem{theorem}{\indent Theorem}
\newtheorem{lemma}{\indent Lemma}
\newtheorem{corollary}{\indent Corollary}

\newtheorem{assumption}{\indent Assumption}
\newtheorem{myremark}{\indent Remark}

\newenvironment{remark}{\begin{myremark}\normalfont}
	{\end{myremark}}

\begin{document}
%
\title{Quantized Primal-Dual Algorithms for Network Optimization with Linear Convergence}
%
%
%

\author{Ziqin Chen, Shu Liang, Li Li and Shuming Cheng
   \thanks{This work was supported in part by National Natural Science Foundation of China under Grant 61903027 and 72171172, by National Key R\&D Program of China under Grants 2018YFE0105000 and 2018YFB1305304, by Shanghai Municipal Science and Technology Major Project under Grant 2021SHZDZX0100, by Shanghai Municipal Science and Technology Fundamental Project under Grant 21JC1405400, and by the Shanghai Municipal Commission of Science and Technology under Grants 1951113210 and 19511132101.}
   \thanks{Ziqin Chen, Shu Liang, Li Li and Shuming Cheng are with the Department
	of Control Science and Engineering, Tongji University, Shanghai, 201210, China (e-mail: cxq0915@tongji.edu.cn; sliang@tongji.edu.cn; lili@tongji.edu.cn; shuming\_cheng@tongji.edu.cn).}}

\maketitle
\vspace{-4em}
\begin{abstract}
 This paper studies the network optimization problem about which a group of agents cooperates to minimize a global function under practical constraints of finite bandwidth communication. Particularly, we propose an adaptive encoding-decoding scheme to handle the constrained communication between agents. Based on this scheme, the continuous-time quantized distributed primal-dual (QDPD) algorithm is developed for network optimization problems. We prove that our algorithms can exactly track an optimal solution to the corresponding convex global cost function at a linear convergence rate. Furthermore, we obtain the relation between communication bandwidth and the convergence rate of QDPD algorithms. Finally, an exponential regression example is given to illustrate our results.
\end{abstract}

\vspace{-0.7em}

%
\IEEEpeerreviewmaketitle

\section{Introduction}
The network optimization problem has been intensely studied in various fields, such as the large-scale machine learning ~\cite{machinelearning1,machinelearning2} and the economic dispatch in power systems~\cite{economicdispatch1,economicdispatch2}. Typically, it can be formulated as minimizing a global function composed of a sum of local functions, each of which is assessed by a local agent in the network. Correspondingly, a large number of distributed algorithms~\cite{DGD,stronglyconvex1,stronglyconvex2,stronglyconvex3,renwei,wenguanghui,liangshu} have been proposed to solve this problem. For example, Nedić \textit{et al.} proposed a distributed gradient descent algorithm with sublinear convergence rate~\cite{DGD}. It was then improved to a linear convergence rate for the strongly convex optimization problem in~\cite{stronglyconvex1,stronglyconvex2,stronglyconvex3}. Further, a distributed primal-dual (PD) algorithm with linear convergence was developed for solving certain convex optimization problem~\cite{liangshu}.

Note that almost all the above distributed algorithms for network optimization require each agent to interchange information with its neighbors. However, if physical constraints, such as sensor battery powers and computing resources, are considered, then it becomes impractical to assume the infinite-bandwidth communication in those algorithms. With regard to the problem of finite-bandwidth communication, one effective tool is the quantization scheme capable of compressing the messages exchanged among agents. Hence, the quantized distributed algorithms have appealed much interest in recent years~\cite{neighborhood1,neighborhood2,yipeng,Qstrongconvex1,Qstrongconvex2,Qconvex1,Qconvex3,Qconvex2,Qcontinuous1,Qcontinuous2,Qcontinuous3}. For example, distributed algorithms with static quantizers can converge to the neighborhood of an optimal solution~\cite{neighborhood1,neighborhood2}, while exact optimal solution can be achieved by using a dynamic quantizer~\cite{yipeng}. A linear convergence rate can be guaranteed via a uniform quantizer in distributed alternating direction method of multipliers~\cite{Qstrongconvex1}, and the same convergence rate can be achieved by the quantized gradient tracking algorithm~\cite{Qstrongconvex2}. However, the strong convexity assumption was imposed~\cite{Qstrongconvex1,Qstrongconvex2}. Relaxing the above strong assumption to be convex, the quantized distributed subgradient descent algorithm with a sublinear convergence rate $O(\ln k/\sqrt{k})$ has been developed to track the optimal solution~\cite{Qconvex1,Qconvex3}. Further, it has been improved to the convergence rate $O(1/k)$ by using the amplified-differential
compression method~\cite{Qconvex2}. Hence, there still lack any distributed linearly convergent algorithm for nonstrongly convex optimization problems over multi-agent networks, under the practical constraint of finite-bandwidth communication. 

In this paper, we propose a novel quantized distributed optimization algorithm. It achieves a linear convergence rate for the convex optimization problems. Also, we present a fully quantitative analysis, showing that the fast convergence rate of the proposed
algorithm will linearly increase the demand for communication bandwidth. The main contributions of this work are summarized as follows. 

\begin{itemize}
\item A novel adaptive encoding and decoding scheme is introduced to deal with the problem of finite bandwidth communication. Following this scheme, we present continuous-time quantized distributed primal-dual (QDPD) algorithms to solve the convex network optimization problem via quantization communication.

\item 
Our QDPD algorithm can converge to an exact optimal solution rather than near optimal solutions in~\cite{neighborhood1,neighborhood2}. Moreover, it admits a linear convergence rate, which improves the sublinear rate for other quantized distributed algorithms. All our results are derived under the mild assumption that the cost functions satisfy metric subregularity, which is weaker than the strong convexity assumption imposed in~\cite{stronglyconvex1,stronglyconvex3,Qstrongconvex1,Qstrongconvex2}.	

\item The relation between the communication bandwidth and the convergence rate is obtained. Particularly, given the communication bandwidth, we can always find a lower bound for the convergence rate of our algorithms. It provides theoretical support for transmission bandwidth settings required by different convergence rates. 
\end{itemize}

The rest of the paper is organized as follows. Section~\ref{problemstatement} briefly formulates the distributed convex optimization problem. In Section~\ref{result}, subjected to the constraint of finite-bandwidth communication, we introduce the adaptive encoder-decoder scheme and present the QDPD algorithm, together with the convergence analysis. Section~\ref{example} provides an example to validate our results, and Section \ref{conclusion} concludes this work.

\textit{Notations:}
A $n$-dimensional vector is written as $x\!\in\!{\mathbb{R}^{n}}$, and $\|x\|$ represents its Euclidean norm. For vectors $x_{1},\!\cdots\!,x_{N}$, denote the stacked column vector by $\boldsymbol{x}\!\!=\![x_{1};\!\cdots\!;x_{N}]$. For a given positive number $a\!\in\!{\mathbb{R}}$, $\lceil a \rceil$ is the smallest integer greater than or equal to $a$. For a set $C\subset \mathbb{R}^{n}$ and a point $x\in{\mathbb{R}^{n}}$, the distance from $x$ to $C$ is defined as $d(x,C)=\inf_{y\in{C}}\|y-x\|$. For a continuously differentiable function $f(x)$, denote by $\nabla f(x)$ its gradient vector. For a map $F:R^{n}\rightarrow R^{n}$, define the epigraph of $F$ as $\text{gph} F=\{(x,y)\in{\mathbb{R}^{n}\times\mathbb{R}^{n}}|y=F(x)\}$.

\section{Problem Formulation}\label{problemstatement}
Consider the network optimization problem in which $N$ agents cooperate to optimize the cost function as
\begin{eqnarray}
\mathop{\text{min}}\limits_{x\in{\mathbb{R}^{n}}} f(x),~~f(x)=\sum_{i=1}^{N}f_{i}(x).
\label{primalproblem}
\end{eqnarray}
Each cost function $f_{i}(x): \mathbb{R}^{n}\rightarrow \mathbb{R}$ is locally processed by an agent $i$ in the network. To solve the above problem, some reasonable assumptions are needed.
\begin{assumption} 
Each local cost function $f_{i}(x)$ is differentiable, convex, and $m_{f_{i}}$-smooth for some $m_{f_{i}}\!>\!0$. 
\end{assumption}
\begin{remark}
Assumption 1 implies the global function $f(x)$ is differentiable, convex, and $m_f$-smooth with $m_{f}\!\!=\!\sum_{i=1}^{N}\!m_{f_{i}}$. It ensures the existence of an optimal solution of problem~\eqref{primalproblem}, commonly used in most relevant works~\cite{Qconvex1,Qconvex2,Qconvex3}.
\end{remark}
Based on graph theory, the interconnection network can be described by an undirected and connected graph $\mathcal{G}\!=\!(\mathcal{V},\mathcal{E})$, where $\mathcal{V}$ is the set of agents with cardinality $|\mathcal{V}|=N$, and $\mathcal{E}\!\subset\! \mathcal{V}\!\times\!\mathcal{V}$ is the corresponding edge set. Denote by $A\!=\![a_{ij}]_{N\!\times\!{N}}$ the adjacency matrix of graph $\mathcal{G}$ with $a_{ij}\!=\!1$ if $(i,j)\!\in\! \mathcal{E}$, otherwise $a_{ij}=0$. For an arbitrary agent $i$, $\mathcal{N}_{i}:=\{j\in{\mathcal{V}\big{|}(i,j)\in \mathcal{E}}\}$ corresponds to the set of its neighbors with cardinality $|\mathcal{N}_{i}|=d_{i}$. Then, the Laplacian matrix is defined as $L_{\mathcal{G}}\!=\!D-A$ with $D=\text{diag}(d_{1},\cdots,d_{N})$, and its eigenvalues are arranged in an ascending order as $0=\sigma_{1}\leq\cdots\leq\sigma_{N}$. 

Using primal dual methods~\cite{kkt1,kkt2,liangshu} enable us to reformulate the above optimization problem~\eqref{primalproblem} as 
\begin{eqnarray}
	\mathop{\text{max}}\limits_{\boldsymbol{\lambda}\in{\mathbb{R}^{Nn}}}\mathop{\text{min}}\limits_{\boldsymbol{x}\in{\mathbb{R}^{Nn}}} \{f(\boldsymbol{x})+\boldsymbol{\lambda}^{T}\boldsymbol{L}_{\mathcal{G}}\boldsymbol{x}\!+\!\boldsymbol{x}^{T}\boldsymbol{L}_{\mathcal{G}}\boldsymbol{x}\},\label{dualproblem}
\end{eqnarray}
where the primal variables are $\boldsymbol{x}\!\!=\![x_{1};\!\cdots\!;x_{N}]$, the dual variables $\boldsymbol{\lambda}\!\!=\![\lambda_{1};\!\cdots\!;\lambda_{N}]$, and the augmented Laplacian matrix $\boldsymbol{L}_{\mathcal{G}}\!\!=\!L_{\mathcal{G}}\otimes I_{n}$. If $(\boldsymbol{x}^*,\boldsymbol{\lambda}^*)$ is a pair of optimal solution for the problem~\eqref{dualproblem}, then it satisfies the KKT condition
\begin{flalign}
	\left\{\begin{array}{l}\boldsymbol{0}\!=\!-\nabla f(\boldsymbol{x}^*)-\!\boldsymbol{L}_{\mathcal{G}}\boldsymbol{x}^*-\boldsymbol{L}_{\mathcal{G}}\boldsymbol{\lambda}^*, \\ 
		\boldsymbol{0}\!=\!\boldsymbol{L}_{\mathcal{G}}\boldsymbol{x}^*.\end{array}\right.\label{kkt1}
\end{flalign}
We further define 
\begin{flalign}
	\mathbf{F}(\boldsymbol{z})\!\!=\!\left[\begin{matrix}
		\nabla f(\boldsymbol{x})\!\!+\!\boldsymbol{L}_{\mathcal{G}} \boldsymbol{x}\!\!+\!\boldsymbol{L}_{\mathcal{G}} {\boldsymbol{\lambda}}\\
		\!\!-\!\boldsymbol{L}_{\mathcal{G}}{\boldsymbol{x}}
	\end{matrix}\right]\!\!\in\!{\mathbb{R}^{2Nn}},\label{Fz}
\end{flalign} 
where $\boldsymbol{z}\!\!=\![\boldsymbol{x};\boldsymbol{\lambda}]\!\in\!{\mathbb{R}^{2Nn}}$, and impose a weak assumption on $\mathbf{F}(\boldsymbol{z})$.
 \begin{assumption}
	$\mathbf{F}(\boldsymbol{z})$ in \eqref{Fz} is $\kappa$-metrically subregular at point $(\boldsymbol{z}^*,0_{2Nn})\in \text{gph}\mathbf{F}$, that is, for some positive constant $\kappa$, there exists an open set $C\supset Z^*$ such that 
	\begin{eqnarray}
		\left\|\mathbf{F}(\boldsymbol{z})\right\|\geq \kappa^{-1}\text{d}(\boldsymbol{z},Z^*),~ \forall \boldsymbol{z}\in C,\label{definition1}
	\end{eqnarray}
where $\boldsymbol{z}^*\!\!=\![\boldsymbol{x}^*;\boldsymbol{\lambda}^*]\!\!\in\! Z^*$ with $Z^*\!\!=\!\{\boldsymbol{z}|\mathbf{F}(\boldsymbol{z})\!\!=\!0\}$.
\end{assumption}
\begin{remark}
Assumption 2 ensures the linear convergence rate of our proposed algorithm, which will be discussed in the subsequent section. Note that it is weaker than the strong convexity assumption in~\cite{Qstrongconvex1,Qstrongconvex2} and could be easily satisfied when the epigraph of the map $\mathbf{F}$ is polyhedral form~\cite{liangshu}. 
\end{remark}

Since the useful quantization communication scheme is used to handle the finite bandwidth constraint imposed in the network problem~\eqref{dualproblem}, that is, each agent broadcasts quantized information by an encoder to and receives quantized information by a decoder from its neighbors, an extra assumption on quantization scheme is further assumed. 

\begin{assumption} For every $i\in{\mathcal{V}}$,
	the optimal solution of problem \eqref{primalproblem} satisfies $\|x^*\|_{\infty}\leq M_{1}$ and the corresponding gradient satisfies $\| \nabla {f}_{i}(x^*)\|_{\infty}\leq M_{2},~i\in{\mathcal{V}}$.
\end{assumption}
 Assumption 3 guarantees that the quantization scheme is persistently excited such that our QDPD algorithm works well even under the non-ideal communication. This is similar to Assumption 2 in~\cite{yipeng} and Assumption 3 in~\cite{Qcontinuous1}, respectively.

\section{Main Result}\label{result}
In this section, to solve problem~\eqref{primalproblem}, or equivalently~\eqref{dualproblem}, we design the QDPD algorithm, based on a novel encoder-decoder scheme to model the quantized communication among agents in Subsection \ref{Ad}. Then, we prove the linear convergence of the proposed QDPD algorithm using the Lyapunov method in Subsection \ref{Ca}. Finally, in Subsection \ref{Ba}, we analyze the relationship between the convergence rate and the required communication bandwidth of the QDPD algorithm.

\vspace{-0.3em}
\subsection{The QDPD Algorithm}\label{Ad}
For the adaptive encoder-decoder scheme used in our QDPD algorithm, both the prime and dual variables are quantized and then transmitted over the network graph $\mathcal{G}$. Thus, it significantly reduces the communication bandwidth, in comparison to the distributed PD algorithm requiring infinity precise transmission~\cite{liangshu}. Specifically, the QDPD algorithm can be decomposed into two procedures:

\textbf{1. Quantized Communication Step}

i) Quantizer: Denote the quantizer by $Q^{L}_{[l,u]}$ where $[l,u],~l\leq u$ characterizes the quantization range and $L+1$ defines the quantization level. As shown in Fig.~1, for a scalar $s\in[l,u]$, the quantized state is written as
\begin{equation}
	Q^{L}_{[l,u]}(s)=\argmin\limits_{0\leq i\leq L}\left\{\left| l+i\frac{u-l}{L}\right|-s\right\}.\label{Q}
\end{equation}
It follows from the above equation that $s$ is encoded into the set $\{0,1,\!\cdots,\!L\}$ and thus needs $\log_{2}(L)$ bits to transmit if no-transmission is required at the zero level. 

ii) Encoder:
We adopt the periodic sampling with period $T$ to sample the continuous-time input of the encoder $z_{j}(t)=[x_{j}(t);\lambda_{j}(t)],~t\geq0$ and thus obtain the discrete variables $z_{j}(kT)=[x_{j}(kT);\lambda_{j}(kT)],~k\geq0$ for any $j\in{\mathcal{V}}$. Then, the output is quantized via the above  quantization scheme into  $q^{z}_{j}(kT)=[q^{x}_{j}(kT),q^{\lambda}_{j}(kT)]$ with
\begin{flalign}
	q^{z}_{j}(kT)=Q^{L}_{\mathcal{P}^{z}_{j}(k)}(z_{j}(kT)),\label{qq}
\end{flalign}
where the quantization range $\mathcal{P}^{z}_{j}(k)$ will be explicitly given in the following subsection. Consequently, each agent $j$ encodes $q^{z}_{j}(kT)$ into a sequence of $2n\log_{2}(L)$ bits as  $q^{z(b)}_{j}(k)$ and then transmits it to its neighbors.

 iii) Decoder:
If an agent $i$ receives the quantized data ${q}_{j}^{z(b)}(k)$, then it can be recovered via the known quantization range $\mathcal{P}^{z}_{j}(k),~j\in{\mathcal{N}_{i}}$, which is easily obtained by using ${q}^{z}_{j}((k-1)T)$.
	\vspace{-1em}
\begin{figure}[H]
	\centering
	\includegraphics[height=1.3cm,width=6cm]{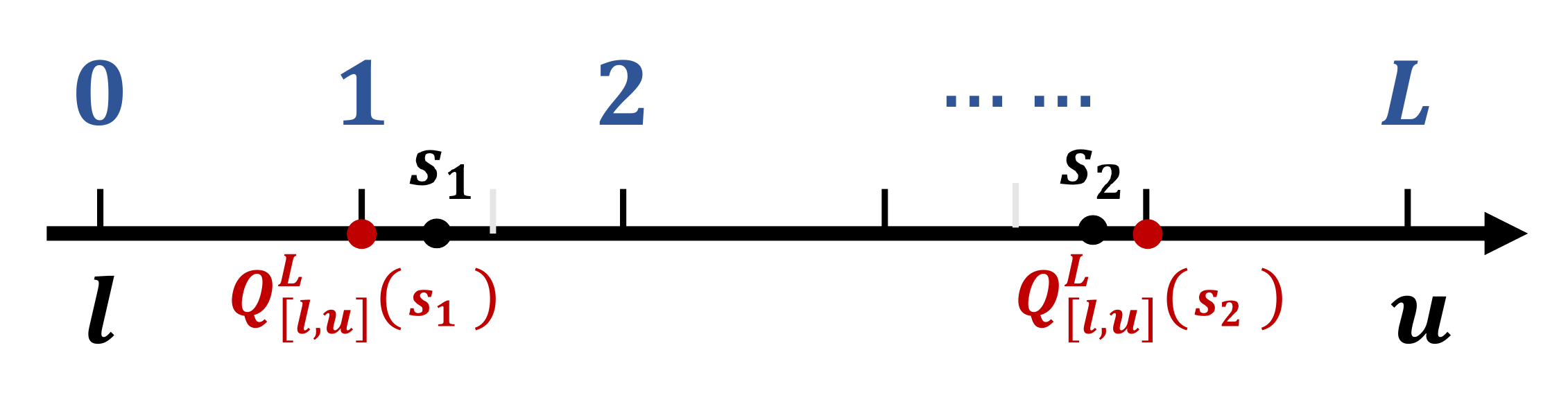}
	\vspace{-0.5em}
	\caption{Quantization scheme}\label{fignetwork}
\end{figure}
	\vspace{-1.5em}
\textbf{2. State Update Step:}

Following from the quantized communication step, each agent updates its states via
\begin{flalign}
		\dot{x}_{i}(t)
		&=\nabla f_{i}\left(x_{i}(t)\right)\!\!-\!\!\!\!\sum\limits_{j\in{\mathcal{N}_{i}}}\left(q^{x}_{i}(t)\!\!-\!q^{x}_{j}(t)\right)\!\!-\!\!\!\!\sum\limits_{j\in{\mathcal{N}_{i}}}\left(q^{\lambda}_{i}(t)
		\!-\!q^{\lambda}_{j}(t)\right),\nonumber\\
		\dot{\lambda}_{i}(t)
		&=\sum\limits_{j\in{\mathcal{N}_{i}}}\left(q^{x}_{i}(t)\!-\!q^{x}_{j}(t)\right).\label{distributed algorithm}
\end{flalign}

Combining with above two steps gives rise to the QDPD algorithm, which is summarized as Algorithm 1. It is developed from the distributed PD method~\cite{liangshu} suitable for the ideal communication scenario. Noting that each agent only transmits the quantized binary sequence at each iteration step, it significantly reduces the source occupancy of communication. 

Given the initial states $\boldsymbol{x}(0)$ and $\boldsymbol{\lambda}(0)$, the key parameters in our QDPD algorithm can be explicitly given as follows.

i) The sampling period $T$ can be any positive constant satisfying
\begin{flalign}
	(e^{\sqrt{\frac{3+\sqrt{5}}{2}}\sigma_{N}T}\!-\!1)(e^{\frac{\eta}{2}T}\!-\!1)\rho\!\leq\! c_{1}\!<\!1,~0<c_{1}<1, \label{T}
\end{flalign}
where $\rho\!=\!\frac{\sqrt{2}(6\sigma_{N}\!+\!2m_{f})\kappa\sqrt{(m_{f}\!+\!4\sigma_{N})(4\sigma_{N}\!+\!11)}}{\eta\sqrt{3\!+\!\sqrt{5}}\sqrt{\sigma_{N}\!-\!\sigma_{N}\kappa^2(\eta\!+\!4)(m_{f}\!+\!4\sigma_{N})}}$ and the parameter $\eta\!=\!\frac{\beta}{\kappa^2(m_{f}+6\sigma_{N})}$.

ii) The dynamic length of quantization range is
\begin{flalign}
	l(k)=l(0)e^{-\frac{\eta}{2}kT}, \label{lk}
\end{flalign}
where $l(0)=\frac{\sqrt{2}c_{2}M_{0}}{\kappa\sigma_{N}}\sqrt{\frac{3\!-\!\kappa^2(3\eta\!+\!4)(m_{f}\!+\!6\sigma_{N})}{Nn(12\sigma_{N}\!+\!33)}}e^{-\frac{\eta}{2}T-\sqrt{\frac{3+\sqrt{5}}{2}}\sigma_{N}T}$ with $M_{0}\!\!=\!\sqrt{2Nn}M\!\!+\!M'$, $M\!\!\geq\! \|\boldsymbol{z}(0)\|_{\infty}$, $M'\!\!\geq\!\!\frac{1_{Nn}^{T}\boldsymbol{\lambda}(0)}{\sqrt{Nn}}\!\!+\!(\frac{\sqrt{Nn}}{\sigma_{2}}\!\!+\!\sqrt{Nn})M_{1}\!\!+\!\sqrt{Nn}M_{2}$, $0<c_{2}<1-c_{1}$ and $\beta\in(0,1)$.

iii) The quantization level $L$ can be any positive constant satisfying
\begin{flalign}
	L\geq\!\left\lceil\max\left\{\frac{2M_{0}}{l(0)},\frac{\sqrt{2Nn}}{c_{2}}e^{\frac{\eta}{2}T\!+\!\sqrt{\frac{3+\sqrt{5}}{2}}\sigma_{N}T}\right\}\right\rceil, \label{L}
\end{flalign}
	
iv) The quantization range is determined by
\begin{flalign}
	\left\{\begin{array}{l}\!\!\!
		\mathcal{P}^{z}_{j}(0)\!\!=\!\!\left[-\frac{Ll(0)}{2}1_{2n},\frac{Ll(0)}{2}1_{2n}\right],\\	
		\!\!\!\mathcal{P}^{z}_{j}(k)\!\!=\!\!\left[{q}^{z}_{j}((k\!-\!1)T)\!-\!\frac{Ll(k)}{2}1_{2n}\!,~{q}^{z}_{j}((\!k\!-\!1)T)\!+\!\frac{Ll(k)}{2}1_{2n}\right].\end{array}\right.\nonumber
\end{flalign} 

\vspace{-0.5em}
\begin{table}[H]
	\renewcommand\arraystretch{1.2}
	\label{dsmb}
	\centering
	\footnotesize
	\begin{tabular}{p{0.95\linewidth}}
		\Xhline{1.0pt}
		{\textbf{Algorithm 1} the QDPD algorithm at agent $i$} \\ \hline 	
		\textbf{Initialization:} 
		\begin{itemize}
			\item Initialize $x_{i}(0)$, $l(0)$, $L$, $T$, $f_{i}(x)$, $\mathcal{P}^{z}_{i}(0)$ and $\mathcal{P}^{z}_{j}(0)$.
			\item Receive the initial quantized data $q_{j}^{z(b)}(0)$ from neighbor $j$ and recover $q_{j}^{z}(0)=Q^{L}_{\mathcal{P}^{z}_{j}(0)}(z_{j}(0))$.  Set $q^{x}_{i}(0)=Q^{L}_{\mathcal{P}^{x}_{i}(0)}(x_{i}(0)).$
			\item Compute dual variables $\lambda_{i}(0)=\sum_{j\in{\mathcal{N}_{i}}}\left(q^{x}_{i}(0)-q^{x}_{j}(0)\right)$ and the corresponding quantized states $q^{\lambda}_{i}(0)=Q^{L}_{\mathcal{P}^{\lambda}_{i}(0)}(\lambda_{i}(0)).$ 
		\end{itemize}
		1: \textbf{For} times $k=0,1,2,\cdots$ \textbf{do}\\
		2: Encode $q_{i}^{z}(kT)$ into $q_{i}^{z(b)}(k)$  and send it to neighbors.\\
		3: Recover the continuous-time signal from the discrete-time signal by
		\vspace{-0.5em}
		$$q_{j}^{z}(t)\!=\!q_{j}^{z}(kT),~kT\leq t<(k+1)T,~j\in{\mathcal{N}_{i}\cup \{i\}}.$$
		4: Update  $x_{i}(t),~\lambda_i(t),~kT\leq t\leq(k+1)T$ by
		\begin{eqnarray}
			\left\{\begin{array}{l}
				\dot{x}_{i}(t)
				\!=\!\nabla f_{i}\left(x_{i}(t)\right)\!\!-\!\!\!\!\sum\limits_{j\in{\mathcal{N}_{i}}}\left(q^{x}_{i}(t)\!\!-\!q^{x}_{j}(t)\right)\!\!-\!\!\!\!\sum\limits_{j\in{\mathcal{N}_{i}}}\left(q^{\lambda}_{i}(t)
				\!-\!q^{\lambda}_{j}(t)\right),\\
				\dot{\lambda}_{i}(t)
				\!=\!\sum\limits_{j\in{\mathcal{N}_{i}}}\left(q^{x}_{i}(t)\!-\!q^{x}_{j}(t)\right)
				.\end{array}\right.\nonumber
		\end{eqnarray}
		5. Set the dynamic length of quantization range $l(k+1)$ from \eqref{lk}.\\
		6: Compute the quantization range
		$$\mathcal{P}_{i}^{z}(k+1)=\left[{q}^{z}_{i}(kT)\!-\!\frac{Ll(k+1)}{2}1_{2n},{q}^{z}_{i}(kT)\!+\!\frac{Ll(k+1)}{2}1_{2n}\right].$$ 
		7: Quantize local state information $$q^{z}_{i}((k+1)T)=Q^{L}_{\mathcal{P}^{z}_{i}(k)}(z_{i}((k+1)T)).$$
		8: Receive $q_{j}^{z(b)}(k+1)$ from neighbors $j\in{\mathcal{N}_{i}}$ and recover $q_{j}^{z}((k+1)T)$ by using known
		$\mathcal{P}^{z}_{j}(k+1)$.\\
		9: \textbf{end} \\
		\Xhline{0.9pt}	
	\end{tabular}
\end{table}

\subsection{Convergence Analysis}\label{Ca}
 We first introduce the following lemma to associate the equilibrium of dynamics~\eqref{distributed algorithm} with the optimal primal-dual solution pair.
\begin{lemma}
Under Assumption 1, the primal variables $x_{i}(t)$ and dual variables $\lambda_{i}(t)$ can be written in a stacked form of
	\begin{equation}
		\left\{\begin{array}{l}\dot{\boldsymbol{x}}(t)=-\nabla f(\boldsymbol{x}(t))-\boldsymbol{L}_{\mathcal{G}} \boldsymbol{q}^{\boldsymbol{x}}(t)-\boldsymbol{L}_{\mathcal{G}}\boldsymbol{q}^{\boldsymbol{\lambda}}(t), \\ 
			\dot{\boldsymbol{\lambda}}(t)=\boldsymbol{L}_{\mathcal{G}}\boldsymbol{q}^{\boldsymbol{x}}(t),\end{array}\right.\label{compactalgorithm}
	\end{equation}
where $\boldsymbol{q}^{\boldsymbol{x}}\!\!=\![q_{1}^{x};\cdots;q_{N}^{x}]\!\in\!{\mathbb{R}^{Nn}}$ and $\boldsymbol{q}^{\boldsymbol{\lambda}}\!\!=\![q_{1}^{\lambda};\cdots;q_{N}^{\lambda}]\!\in\!{\mathbb{R}^{Nn}}$. Then, the equilibrium of dynamics~\eqref{compactalgorithm} is a pair of optimal solutions of the problem \eqref{dualproblem}.
\end{lemma}

Since Lemma 1 is similar to Lemma 4 of the work~\cite{liangshu}, the proof is referred to that reference. 

Then, we prove the main result that our QDPD algorithm converges to an optimal solution linearly.
\begin{theorem}
Under Assumptions 1-3, the QDPD algorithm ensures that the states of all agents converge to an optimal solution to the problem \eqref{primalproblem} at rate $O(\gamma^{-t})$ with $\gamma=e^{\frac{\eta}{2}}>1$. 
\end{theorem}
\vspace{-1em}
\begin{proof}
First, we construct a Lyapunov candidate function as
\begin{equation}
	V(\boldsymbol{z})=4\sigma_{N}V_{1}(\boldsymbol{z})+V_{2}(\boldsymbol{z}),\label{V}
\end{equation}
where
\begin{eqnarray}
	\left\{\begin{array}{l}
V_{1}(\boldsymbol{z})=\frac{1}{2}\|\boldsymbol{z}-\boldsymbol{z}^*\|^2,\nonumber\\
V_{2}(\boldsymbol{z})=f(\boldsymbol{x})-f(\boldsymbol{x}^*)+\frac{1}{2}\boldsymbol{x}^{T}\boldsymbol{L}_{\mathcal{G}}\boldsymbol{x}+\boldsymbol{\lambda}^{T}\boldsymbol{L}_{\mathcal{G}}\boldsymbol{x}.\end{array}\right.\nonumber
\end{eqnarray}
Note that if $V(\boldsymbol{z})$ is always positive and decays to zero at a linear convergence rate, then $\boldsymbol{z}(t)$ will linearly converge to $\boldsymbol{z}^*$. We now prove that $V(\boldsymbol{z})$ does satisfy these properties.

 Multiplying $(\boldsymbol{x}-\boldsymbol{x}^*)^T$ left by~\eqref{kkt1} and using the relation $(\boldsymbol{x}^*)^{T}\boldsymbol{L}_{\mathcal{G}}=0_{Nn}^{T}$ yields
\begin{eqnarray}
	\left\{\begin{array}{l}
		(\boldsymbol{x}-\boldsymbol{x}^*)^T\nabla f(\boldsymbol{x}^*)=-(\boldsymbol{x}-\boldsymbol{x}^*)^T\boldsymbol{L}_{\mathcal{G}}\boldsymbol{\lambda}^* ,\\ 
		\boldsymbol{x}^T\boldsymbol{L}_{\mathcal{G}}\boldsymbol{\lambda}=(\boldsymbol{x}-\boldsymbol{x}^*)^{T}\boldsymbol{L}_{\mathcal{G}}\boldsymbol{\lambda}.\end{array}\right.\nonumber
\end{eqnarray}
This immediately leads to
\begin{flalign}
\begin{split}
	V_{2}(\boldsymbol{z})&=f(\boldsymbol{x})-f(\boldsymbol{x}^*)-(\boldsymbol{x}-\boldsymbol{x}^*)^{T}\nabla f(\boldsymbol{x}^*)\!\!+\!\frac{1}{2}(\boldsymbol{x}\!\!-\!\boldsymbol{x}^*)^{T}\nonumber\\
	&~~~\boldsymbol{L}_{\mathcal{G}}(\boldsymbol{x}\!\!-\!\boldsymbol{x}^*)\!\!+\!(\boldsymbol{x}\!\!-\!\boldsymbol{x}^*)^{T}\boldsymbol{L}_{\mathcal{G}}(\boldsymbol{\lambda}\!\!-\!\boldsymbol{\lambda}^*).\nonumber
\end{split}&
\end{flalign}
It follows further from the convexity of $f(\boldsymbol{x})$ that $f(\boldsymbol{x})\!-\!f(\boldsymbol{x}^*)\!-\!(\boldsymbol{x}\!-\!\boldsymbol{x}^*)^{T}\nabla f(\boldsymbol{x}^*)\!\!\geq\! 0$. And there is $\frac{1}{2}(\boldsymbol{x}\!-\!\boldsymbol{x}^*)^{T}\boldsymbol{L}_{\mathcal{G}}(\boldsymbol{x}\!-\!\boldsymbol{x}^*)\!\geq\! 0$. as $\boldsymbol{L}_{\mathcal{G}}$ is positive semidefinite. Hence, we can obtain
\begin{eqnarray}
	V_{2}(\boldsymbol{z})
	\!\!\geq\!\!-\frac{\sigma_{N}}{2}(\|\boldsymbol{x}\!-\!\boldsymbol{x}^*\|^2\!\!+\!\|\boldsymbol{\lambda}\!\!-\!\boldsymbol{\lambda}^*\|^2)\!\geq\!-\frac{\sigma_{N}}{2}\|\boldsymbol{z}\!-\!\boldsymbol{z}^*\|^2,\label{v2u}
\end{eqnarray}
and prove that $V(\boldsymbol{z})\geq \frac{3\sigma_{N}}{2}\|\boldsymbol{z}-\boldsymbol{z}^*\|^2\geq0.$

Then, we prove that $V(\boldsymbol{z})$ converges to zero linearly. Denote the quantization error by $\boldsymbol{e}\!\!=\![\boldsymbol{e}_{\boldsymbol{x}};\boldsymbol{e}_{\boldsymbol{\lambda}}]$ with $\boldsymbol{e}_{\boldsymbol{x}}\!\!=\!\boldsymbol{x}-\boldsymbol{q}^{\boldsymbol{x}}$ and $\boldsymbol{e}_{\boldsymbol{\lambda}}\!\!=\!\boldsymbol{\lambda}-\boldsymbol{q}^{\boldsymbol{\lambda}}$. Combining with Eqs.~\eqref{kkt1}, \eqref{compactalgorithm}, and $(\boldsymbol{x}\!-\!\boldsymbol{x}^*)^T(\nabla f(\boldsymbol{x})\!-\!\nabla f(\boldsymbol{x}^*))\!\geq\! 0$ yields 
\begin{flalign}
	\dot{V}_{1}(\boldsymbol{z})\!\!
&=\!\!-(\boldsymbol{x}\!-\!\boldsymbol{x}^*)^{T}(\nabla f(\boldsymbol{x})\!\!-\!\nabla f(\boldsymbol{x}^*))\!\!+\!\boldsymbol{x}^T\boldsymbol{L}_{\mathcal{G}}\boldsymbol{e_{x}}\!\!-\!\boldsymbol{x}^T\boldsymbol{L}_{\mathcal{G}}\boldsymbol{x}\nonumber\\
	\!\!&~~~+\!\!\boldsymbol{x}^T\boldsymbol{L}_{\mathcal{G}}\boldsymbol{e_{\lambda}}\!\!-\!(\boldsymbol{\lambda}-\boldsymbol{\lambda}^*)^T\boldsymbol{L}_{\mathcal{G}}\boldsymbol{e_{x}}.\nonumber\\
	\!\!&\leq\!\!-\sigma_{N}^{-1}\|\boldsymbol{L}_{\mathcal{G}}\boldsymbol{x}\|\!\!+\!\boldsymbol{x}^T\boldsymbol{L}_{\mathcal{G}}(\boldsymbol{e_{x}}\!\!+\!\!\boldsymbol{e_{\lambda}})\!\!-\!\!(\boldsymbol{\lambda}\!-\!\boldsymbol{\lambda}^*)^T\boldsymbol{L}_{\mathcal{G}}\boldsymbol{e_{x}}.\nonumber
\end{flalign} 
Here we have used the relations $\boldsymbol{e}_{\boldsymbol{x}^*}\!=\!0_{Nn}$ and $\boldsymbol{e}_{\boldsymbol{\lambda}^*}\!=\!0_{Nn}$. Further, noting $\boldsymbol{x}^T\boldsymbol{L}_{\mathcal{G}}(\boldsymbol{e_{x}}\!\!+\!\!\boldsymbol{e_{\lambda}})\leq\|\boldsymbol{L}_{\mathcal{G}}\boldsymbol{x}\|\|\boldsymbol{e}\|\leq\frac{\varepsilon_{1}}{2\sigma_{N}}\|\boldsymbol{L}_{\mathcal{G}}\boldsymbol{x}\|^2\!\!+\!\frac{\sigma_{N}}{2\varepsilon_{1}}\|\boldsymbol{e}\|^2$ for $0<\varepsilon_{1}<1$, we obtain
\begin{flalign}
	\dot{V}_{1}(\boldsymbol{z})
	\!\!&\leq\!\!-\!\frac{(1\!-\!\varepsilon_{1})}{\sigma_{N}}\|\!\boldsymbol{L}_{\mathcal{G}}\boldsymbol{x}\!\|^2\!\!+\!\frac{\sigma_{N}}{2\varepsilon_{1}}\|\boldsymbol{e}\|^2\!-\!(\boldsymbol{\lambda}\!\!-\!\boldsymbol{\lambda}^*)^{T}\boldsymbol{L}_{\mathcal{G}}\boldsymbol{e}_{\boldsymbol{x}}.\label{dotV11}
\end{flalign}

Moreover, there is
\begin{flalign}
	\dot{V}_{2}(\boldsymbol{z})\!\!
	&=\!\!-\|\nabla f(\boldsymbol{x})\!\!+\!\boldsymbol{L}_{\mathcal{G}}\boldsymbol{x}\!\!+\!\boldsymbol{L}_{\mathcal{G}}\boldsymbol{\lambda}\|^2\!\!+\!(\nabla f(\boldsymbol{x})\!\!+\!\boldsymbol{L}_{\mathcal{G}}\boldsymbol{x}\!\!+\!\boldsymbol{L}_{\mathcal{G}}\boldsymbol{\lambda})^{T}\nonumber\\
	\!\!&~~~~\boldsymbol{L}_{\mathcal{G}}(\boldsymbol{e}_{\boldsymbol{x}}\!\!+\!\boldsymbol{e}_{\boldsymbol{\lambda}})\!\!+\!\|\boldsymbol{L}_{\mathcal{G}}\boldsymbol{x}\|^2-(\boldsymbol{L}_{\mathcal{G}}\boldsymbol{x})^T\boldsymbol{L}_{\mathcal{G}}\boldsymbol{e}_{\boldsymbol{x}}.\nonumber
\end{flalign}
And, using $(\nabla f(\boldsymbol{x})\!\!+\!\boldsymbol{L}_{\mathcal{G}}\boldsymbol{x}\!\!+\!\boldsymbol{L}_{\mathcal{G}}\boldsymbol{\lambda})^{T}\boldsymbol{L}_{\mathcal{G}}(\boldsymbol{e}_{\boldsymbol{x}}\!\!+\!\boldsymbol{e}_{\boldsymbol{\lambda}})\!\!\leq\!\|\nabla f(\boldsymbol{x})\!+\!\boldsymbol{L}_{\mathcal{G}}\boldsymbol{x}\!+\!\boldsymbol{L}_{\mathcal{G}}\boldsymbol{\lambda}\|\|\boldsymbol{L}_{\mathcal{G}}\|\|\boldsymbol{e}\|\!\!\leq\!\frac{\varepsilon_{2}}{2}\|\nabla f(\boldsymbol{x})\!+\!\boldsymbol{L}_{\mathcal{G}}\boldsymbol{x}\!+\!\boldsymbol{L}_{\mathcal{G}}\boldsymbol{\lambda}\|^2\!+\!\frac{1}{2\varepsilon_{2}}\|\boldsymbol{L}_{\mathcal{G}}\|^2\|\boldsymbol{e}\|^2$ for $0<\varepsilon_{2}<1$ gives rise to
\begin{flalign}
	\dot{V}_{2}(\boldsymbol{z})
	\!\!&\leq\!\!-(1\!\!-\!\varepsilon_{2})\|\nabla f(\boldsymbol{x})\!\!+\!\boldsymbol{L}_{\mathcal{G}}\boldsymbol{x}\!\!+\!\boldsymbol{L}_{\mathcal{G}}\boldsymbol{\lambda}\|^2\!\!+\!(2\varepsilon_{2})^{-1}\|\boldsymbol{L}_{\mathcal{G}}\|^2\|\boldsymbol{e}\|^2\nonumber\\
	\!\!&~~+\!\!\|\boldsymbol{L}_{\mathcal{G}}\boldsymbol{x}\|^2\!\!-\!(\boldsymbol{L}_{\mathcal{G}}\boldsymbol{x})^T\boldsymbol{L}_{\mathcal{G}}\boldsymbol{e}_{\boldsymbol{x}}.\label{V222}
\end{flalign}
Consequently, with \eqref{V}, we are able to derive
\begin{align}
	\dot{V}(\boldsymbol{z})\!\!
	\!\!&\leq\!\!-2\|\boldsymbol{L}_{\mathcal{G}}\boldsymbol{x}\|^2\!\!+\!4\sigma_{N}^2\|\boldsymbol{e}\|^2\!\!-\!4\sigma_{N}(\boldsymbol{\lambda}\!\!-\!\boldsymbol{\lambda}^*)^T\boldsymbol{L}_{\mathcal{G}}\boldsymbol{e}_{\boldsymbol{x}}\!\!-\!\frac{1}{2}\|\nabla f(\boldsymbol{x})\nonumber\\
	\!\!&~~+\!\boldsymbol{L}_{\mathcal{G}}\boldsymbol{x}\!\!+\!\boldsymbol{L}_{\mathcal{G}}\boldsymbol{\lambda}\|^2\!\!+\!\sigma_{N}^2\|\boldsymbol{e}\|^2\!\!+\!\frac{3}{2}\|\boldsymbol{L}_{\mathcal{G}}\boldsymbol{x}\|^2\!\!+\!\frac{1}{2}\|\boldsymbol{L}_{\mathcal{G}}\boldsymbol{e}_{\boldsymbol{x}}\|^2,\nonumber\\
	\!\!&\leq\!\!-\frac{1}{2}\|\boldsymbol{F}(\boldsymbol{z})\|^2\!\!+\!5\sigma_{N}^2\|\boldsymbol{e}\|^2\!\!+\!2\sigma_{N}\|\boldsymbol{\lambda}\!-\!\boldsymbol{\lambda}^*\|^{2}\!+\!(2\sigma_{N}\!+\!\frac{1}{2})\nonumber\\
	\!\!&~~~\|\boldsymbol{L}_{\mathcal{G}}\boldsymbol{e}_{\boldsymbol{x}}\|^{2},\label{VD}
\end{align}
for $\varepsilon_1=\varepsilon_2=1/2$. The first inequality follows from $\!\!-\!(\boldsymbol{L}_{\mathcal{G}}\boldsymbol{x})^T\boldsymbol{L}_{\mathcal{G}}\boldsymbol{e}_{\boldsymbol{x}}\!\leq\!\frac{1}{2}\|\boldsymbol{L}_{\mathcal{G}}\boldsymbol{x}\|^2\!+\!\frac{1}{2}\|\boldsymbol{e}_{\boldsymbol{x}}\|^2$. Since
$\|\boldsymbol{\lambda}-\boldsymbol{\lambda}^*\|^{2}\leq\frac{2V(\boldsymbol{z})}{3\sigma_{N}}$ and $\|\boldsymbol{L}_{\mathcal{G}}\boldsymbol{e}_{\boldsymbol{x}}\|^2\leq\sigma_{N}^2\|\boldsymbol{e}\|^2$,  Eq.~\eqref{VD} can be further upper bounded by
\begin{flalign}
	\dot{V}(\boldsymbol{z})\!\!\leq\!-\!\frac{1}{2}\|\!\boldsymbol{F}(\boldsymbol{z})\|^{2}\!\!+\!\!(2\sigma_{N}^3\!+\!\frac{11}{2}\sigma_{N}^2)\|\boldsymbol{e}(t)\|^{2}\!\!+\!\!\frac{4V(\boldsymbol{z})}{3}.\label{VVV}
\end{flalign}
This indicates the decaying property of $V(\boldsymbol{z})$ is related to $\|\boldsymbol{e}(t)\|$. Hence, the linear convergence of $V(\boldsymbol{z})$ is essentially equivalent to the following inequalities
\begin{equation}
	\left\{\begin{array}{l}
		V(\boldsymbol{z})\!\leq a(t),~0\leq t\leq kT,~k\in{\mathbb{N}},\\
		\|\boldsymbol{e}(kT)\|\!\leq\!c_{2}e^{-\sqrt{\frac{3+\sqrt{5}}{2}}\sigma_{N}T}b(kT), \\
		\|\boldsymbol{e}(t)\|\!< b(t),~~0\leq t<kT, \end{array}\right.\label{lem22}
\end{equation}
where
\begin{flalign}
	a(t)\!\!&=\!\!\frac{m_{f}\!+\!6\sigma_{N}}{2}M_{0}^2e^{-\eta t},\label{at}\\
	b(t)\!\!&=\!\!\frac{M_{0}}{\kappa\sigma_{N}}\sqrt{\frac{3\!-\!\kappa^2(3\eta+4)(m_{f}\!+\!6\sigma_{N})}{12\sigma_{N}\!+\!33}}e^{-\frac{\eta}{2}(\lfloor\frac{t}{T}\rfloor\!+\!1)T}.\label{bt}
\end{flalign}
Inequality $V(\boldsymbol{z})\leq a(t)$ in \eqref{lem22} means that $V(\boldsymbol{z})$ converges to zero linearly. The corresponding proof is deferred to the Appendix. 

Finally, via the convergence rate of $V(\boldsymbol{z})$ being $e^{-\eta t}$, we conclude that $\boldsymbol{z}(t)$ converges to $\boldsymbol{z}^*$ and $x_{i}(t) $ to $x^*$ for $i\in{\mathcal{V}}$ at a linear convergence rate. This rate can be explicitly given as $O(\gamma^{-t})$ with $\gamma=e^{\frac{\eta}{2}}>1$.
\end{proof}

Theorem 1 establishes the linear convergence of the QDPD algorithm even with the communication constraints of quantization. As a matter of fact, the cumulative quantization errors from quantized communication bring a difficulty to convergence analysis. For eliminating the effect of quantization errors, we introduce a special decaying quantization length strategy $l(k)$ in our QDPD algorithms. This well-designed $l(k)$ ensures the exponential convergence of quantization errors and contributes to the linear convergence of the QDPD algorithm.

\subsection{Bandwidth Analysis}\label{Ba}
In this subsection, we analyze the relationship between the bandwidth and the convergence rate. To characterize this relationship, we need to adjust the convergence rate dynamically. Hence, a positive gain parameter $\alpha$ is introduced into the QDPD algorithm and the dynamics \eqref{distributed algorithm} is redesigned as follows,
\begin{eqnarray}
	\left\{\begin{array}{l}
		\dot{x}_{i}(t)
		\!=\!\alpha \nabla f_{i}\left(x_{i}(t)\right)\!-\!\alpha\sum\limits_{j\in{\mathcal{N}_{i}}}\left(q^{x}_{i}(t)\!\!-\!q^{x}_{j}(t)\right)\!-\!\\
		~~~~~~~~~\alpha\sum\limits_{j\in{\mathcal{N}_{i}}}\left(q^{\lambda}_{i}(t)
		\!-\!q^{\lambda}_{j}(t)\right),\\
		\dot{\lambda}_{i}(t)
		\!=\!\alpha\sum\limits_{j\in{\mathcal{N}_{i}}}\left(q^{x}_{i}(t)\!-\!q^{x}_{j}(t)\right),~\alpha>0.
	\end{array}\right.\label{rcompactalgorithm}
\end{eqnarray}
When $\alpha\!=\!1$, \eqref{rcompactalgorithm} matches Algorithm 1. In fact, the introduced gain $\alpha$ does not affect the qualitative result of linear convergence for the QDPD algorithm with dynamic \eqref{rcompactalgorithm}, but has a quantitative impact on the convergence rate. We present the following corollary to illustrate this point.
\begin{corollary} 
Consider the QDPD algorithm with modified dynamics \eqref{rcompactalgorithm}, in which some parameters are revised as 

	i) $T_{\alpha}$ satisfies 
	\begin{eqnarray}
		(e^{\alpha\sqrt{\frac{3+\sqrt{5}}{2}}\sigma_{N}T_{\alpha}}\!-\!1)(e^{\frac{\alpha\eta}{2}T_{\alpha}}\!-\!1)\rho\alpha^{-1}<c_{1}<1\label{T2}.
	\end{eqnarray}

	ii) $l_{\alpha}(k)$ satisfies $$l_{\alpha}(k)=l_{\alpha}(0)e^{-\frac{\alpha\eta}{2}kT_{\alpha}},$$
	with $l_{\alpha}(0)\!\!=\!\frac{\sqrt{2}c_{2}M_{0}}{\kappa\sigma_{N}}\!\sqrt{\!\!\frac{3\!\!-\!\kappa^2(3\eta\!+\!4)(m_{f}\!+\!6\sigma_{N})}{Nn(12\sigma_{N}\!+\!33)}}e^{\!-\frac{\alpha\eta}{2}T_{\alpha}\!-\!\alpha\sqrt{\frac{3+\sqrt{5}}{2}}\!\sigma_{N}\!T_{\alpha}}$. 
	
	iii) $L_{\alpha}$ satisfies
	\begin{eqnarray}
		L_{\alpha}\geq\!\left\lceil\max\left\{\frac{2M_{0}}{l_{\alpha}(0)},\frac{\sqrt{2Nn}}{c_{2}}e^{\frac{\alpha\eta}{2}T_{\alpha}\!+\!\alpha\sqrt{\frac{3+\sqrt{5}}{2}}\sigma_{N}T_{\alpha}}\right\}\right\rceil.\label{L2}
	\end{eqnarray}

  Under Assumptions 1-3, the states of all agents converge to an optimal solution to the problem \eqref{primalproblem} at rate $O(\gamma_{\alpha}^{-t})$ with $\gamma_{\alpha}=e^{\frac{\alpha\eta}{2}}>1$. 
\end{corollary}
Since Corollary 1 is similar to Theorem 1, the proof is omitted here. In Corollary 1, the upper bound of the minimum bandwidth $\mathcal{B}_{\alpha}=\log_{2}(L_{\alpha})/T_{\alpha}$ and the convergence rate $\gamma_{\alpha}$ are dynamically adjusted by the gain $\alpha$. Next, we discuss their corresponding relationship in the following theorem.

\begin{theorem}
Under Assumptions 1-3, the following properties hold.

i) The QDPD algorithm with \eqref{rcompactalgorithm} linearly converges to an optimal solution of problem \eqref{primalproblem} under any positive bandwidth.

ii) The relationship between the communication bandwidth $\mathcal{B}$ and the convergence rate $\gamma_{\alpha}$ satisfies
\begin{eqnarray}
\mathcal{B}\leq C_{0}\ln\gamma_{\alpha}+C_{1},\label{Br}
\end{eqnarray}
where the constants $C_{0}$ and $C_{1}$ are chosen as 
\begin{flalign}
C_{0}&=(\log_{2}e)(1+\frac{\sqrt{6+\sqrt{5}}\sigma_{N}}{\eta})+\frac{\sqrt{24+8\sqrt{5}}\sigma_{N}+2\eta}{\eta\ln\rho_{0}}\nonumber\\
&~~(\log_{2}(2Nn)-2\log_{2}c_{2}),\nonumber\\
C_{1}&=(\frac{1}{2}\log_{2}(2Nn)-\log_{2}c_{2})\frac{2c_{2}}{\rho\rho_{0}\eta},~\rho_{0}>1.\nonumber
\end{flalign}
\end{theorem}
\begin{proof}
i) According to Corollary 1,  \eqref{rcompactalgorithm} ensures the linear convergence of the QDPD algorithm. Next, we discuss the parameters required in Corollary 1. 

Assume that $\eta$ is a fixed constant. In this case, for any $T_{\alpha}>0$, there exists a sufficiently small $\alpha$ such that \eqref{T2} holds. Then, the quantization level $L_{\alpha}$ is estimated via \eqref{L2} as follows,
\begin{eqnarray}
\lim_{\alpha\rightarrow0}(\sqrt{2Nn}/c_{2})e^{\frac{\alpha\eta}{2}T_{\alpha}\!+\!\alpha\sqrt{\frac{3+\sqrt{5}}{2}}\sigma_{N}T_{\alpha}}=\sqrt{2Nn}/c_{2},\nonumber
\end{eqnarray}
which means that the transmitted information could be represented by $\log_{2}(\lceil\max\{\frac{2M_{0}}{l(0)},\frac{\sqrt{2Nn}}{c_{2}}\}\rceil)$ bits data at each instant. Since $T_{\alpha}$ can be chosen as any positive constant, the upper bound of the minimum bandwidth $\mathcal{B}_{\alpha}=\log_{2}(\lceil\max\{\frac{2M_{0}}{l(0)},\frac{\sqrt{2Nn}}{c_{2}}\}\rceil)/T_{\alpha}$ 
can be any positive constant. It implies for any positive bandwidth, the QDPD algorithm with \eqref{rcompactalgorithm} achieves linear convergence.

ii) For a fixed convergence rate $\gamma_{\alpha}$, we compute the corresponding upper bound of the minimum bandwidth $\mathcal{B}_{\alpha}$. 

By choosing that
\begin{eqnarray}
T_{\alpha}=\frac{1}{\frac{\sqrt{24+8\sqrt{5}}\sigma_{N}+2\eta}{2\eta\ln\rho_{0}}\ln\gamma_{\alpha}\!+\!\frac{2c_{1}}{\rho\rho_{0}\eta}},\label{T3}~\rho_{0}>1,\label{Ts}
\end{eqnarray}
Following from $e^{x}-1\leq e^{x}$ and $e^{x}-1\leq xe^{x}$, one has that
\begin{eqnarray}
(e^{\alpha\sqrt{\frac{3+\sqrt{5}}{2}}\sigma_{N}T_{\alpha}}\!\!-\!1)(e^{\frac{\alpha\eta}{2}T_{\alpha}}\!\!-\!1)/\alpha
\!\!\leq\!\frac{\eta T_{\alpha}}{2} e^{\alpha\sqrt{\frac{3+\sqrt{5}}{2}}\sigma_{N}T_{\alpha}\!+\!\frac{\alpha\eta T_{\alpha}}{2}}.\nonumber
\end{eqnarray}
Invoking $T_{\alpha}\!\leq\!\frac{2\eta\ln\rho_{0}}{(\sqrt{24+8\sqrt{5}}\sigma_{N}\!+\!2\eta)\ln\gamma_{\alpha}}$ from \eqref{Ts} yields that
\begin{eqnarray}
e^{\alpha\sqrt{\frac{3+\sqrt{5}}{2}}\sigma_{N}T_{\alpha}\!+\!\frac{\alpha\eta}{2}T_{\alpha}}\!<\!\rho_{0},\label{Tb2}
\end{eqnarray}
Following $T_{\alpha}\!\leq\!\frac{\rho\rho_{0}\eta}{2c_{1}}$ from \eqref{Ts}, one can obtain that
$$(e^{\alpha\sqrt{\frac{3+\sqrt{5}}{2}}\sigma_{N}T_{\alpha}}\!-\!1)(e^{\frac{\alpha\eta}{2}T_{\alpha}}\!-\!1)\rho\alpha^{-1}\leq c_{1}<1.$$ 
Thus, $T_{\alpha}$ in \eqref{Ts} satisfies \eqref{T2}.

Noting the quantization levels at initial time has no effect on the value of the required communication bandwidth, we compute $\mathcal{B}_{\alpha}$ via \eqref{L2} as
\begin{flalign}
&\mathcal{B}_{\alpha}=\frac{\log_{2}(L_{\alpha})}{T_{\alpha}}= \frac{\log_{2}(\frac{\sqrt{2Nn}}{c_{2}}e^{\frac{\alpha\eta}{2}T_{\alpha}\!+\!\alpha\sqrt{\frac{3+\sqrt{5}}{2}}\sigma_{N}T_{\alpha}})}{T_{\alpha}},\nonumber\\
&=\! \frac{\frac{1}{2}\log_{2}(2Nn)\!-\!\log_{2}c_{2}}{T_{\alpha}}\!+\!\log_{2}e(\ln \gamma_{\alpha}\!+\!\alpha\sqrt{\frac{3\!+\!\sqrt{5}}{2}}).\label{B0}
\end{flalign}
Substituting \eqref{Ts} into \eqref{B0}, we conclude that $\mathcal{B}_{\alpha}=C_{0}\ln(\gamma_{\alpha})+C_{1}$. Hence, \eqref{Br} holds.
\end{proof}

Note that property i) provides a necessary and sufficient condition on the required bandwidth for linear convergence of the QDPD algorithm. Specifically, property i) provides its sufficiency. Its necessity can be deduced directly by Shannon's rate-distortion theory. That is, if the quantized distributed algorithm achieves linear convergence with the rate $\gamma_{\alpha}$, the bandwidth $\mathcal{B}_{\alpha}>\gamma_{\alpha}\log_{2}e>0$. 

Property ii) offers the upper bound of the minimum bandwidth. Hence, there may exist a $\mathcal{B}\!<\!\mathcal{B}_{\alpha}$ such that the QDPD algorithm achieves linear convergence with the fixed rate $\gamma_{\alpha}$. Conversely, for a fixed bandwidth $\mathcal{B}_{\alpha}$, the QDPD algorithm can always guarantee the linear convergence rate $\gamma_{\alpha}$. 

\section{Example}\label{example}
In this section, we use an exponential regression example to verify our result. Consider a training set $S\!=\!\{(a_{i},b_{i},c_{i},d_{i})^T\!\in\!\mathbb{R}^4\}$ for $i\!\in\!\{1,\cdots,12\}$, which is given in Table I.

We aim to learn a parameter $x$ to minimize problem \eqref{primalproblem}, whose local cost functions are designed as follows,
\begin{equation}
f_{i}(x)= \left\{\begin{array}{l}
	a_{i}(x-b_{i})^2,~ \text{if}~ x\geq b_{i},\\
	c_{i}(x+d_{i})^2,~ \text{if}~ x\leq -d_{i},\\
	0,~\text{otherwise},
\end{array}\right.
\end{equation}
which satisfies Assumption 1, and the corresponding $\mathbf{F}$ satisfies Assumption 2 based on Lemma 2 in \cite{liangshu}. Take 12 agents cooperatively to learn the optimal solution $x^*\!\in\![0,1]$. The network graph is described as a ring and the information transmission among agents is sampled and quantized. The parameters of the quantization scheme are chosen as: i) the sampling period $T=0.05 \text{s}$; ii) the dynamic length of quantization range $l(k)=0.8e^{-0.1k}$; iii) the number of quantization levels $L+1=68$. The initial states of all agents are set as $\boldsymbol{x}(0)=[-9;4;-9;-9;0;-8;6;6;4;-7;3;0]$.
\begin{table}
	\renewcommand\arraystretch{1.2}
	\caption{}
	\label{dsmb}
	\centering
	\footnotesize
	\begin{tabular}{p{0.95\linewidth}}
		\Xhline{1.0pt}
		{\textbf{i}~~~~~~~\textbf{a}~~~~~~~\textbf{b}~~~~~~~\textbf{c}~~~~~~~\textbf{d}~~~~~~~~~~\textbf{i}~~~~~~~\textbf{a}~~~~~~~\textbf{b}~~~~~~~\textbf{c}~~~~~~~\textbf{d}} \\ \hline 
		{\textbf{1}~~~~~~~{0.1}~~~~~{0.5}~~~~~{1.0}~~~~~{9.0}~~~~~~~\textbf{7}~~~~~~~{0.5}~~~~~{0.4}~~~~~{1.0}~~~~~{5.0}} \\ \hline 
		{\textbf{2}~~~~~~~{0.3}~~~~~{0.2}~~~~~{3.0}~~~~~{3.0}~~~~~~~\textbf{8}~~~~~~~{0.6}~~~~~{1.0}~~~~~{7.0}~~~~~{5.0}} \\ \hline
		{\textbf{3}~~~~~~~{0.8}~~~~~{0.5}~~~~~{3.0}~~~~~{7.0}~~~~~~~\textbf{9}~~~~~~~{0.2}~~~~~{0.0}~~~~~{5.0}~~~~~{9.0}} \\ \hline  
		{\textbf{4}~~~~~~~{0.0}~~~~~{0.6}~~~~~{7.0}~~~~~{2.0}~~~~~~~\textbf{10}~~~~~~{0.5}~~~~~{0.9}~~~~~{8.0}~~~~~{6.0}} \\ \hline 
		{\textbf{5}~~~~~~~{0.9}~~~~~{0.7}~~~~~{1.0}~~~~~{0.0}~~~~~~~\textbf{11}~~~~~~{1.0}~~~~~{0.9}~~~~~{7.0}~~~~~{6.0}} \\ \hline 
		{\textbf{6}~~~~~~~{0.7}~~~~~{0.4}~~~~~{7.0}~~~~~{7.0}~~~~~~~\textbf{12}~~~~~~{0.5}~~~~~{0.8}~~~~~{9.0}~~~~~{9.0}} \\ \hline 
	\end{tabular}
\end{table}

i) The trajectories of $\boldsymbol{x}(t)$ and $\boldsymbol{\lambda}(t)$ are shown in Fig.2, which demonstrates the convergence of the QDPD algorithm. 

\begin{figure}
	\centering
	\subfigure{\includegraphics[height=3cm,width=7cm]{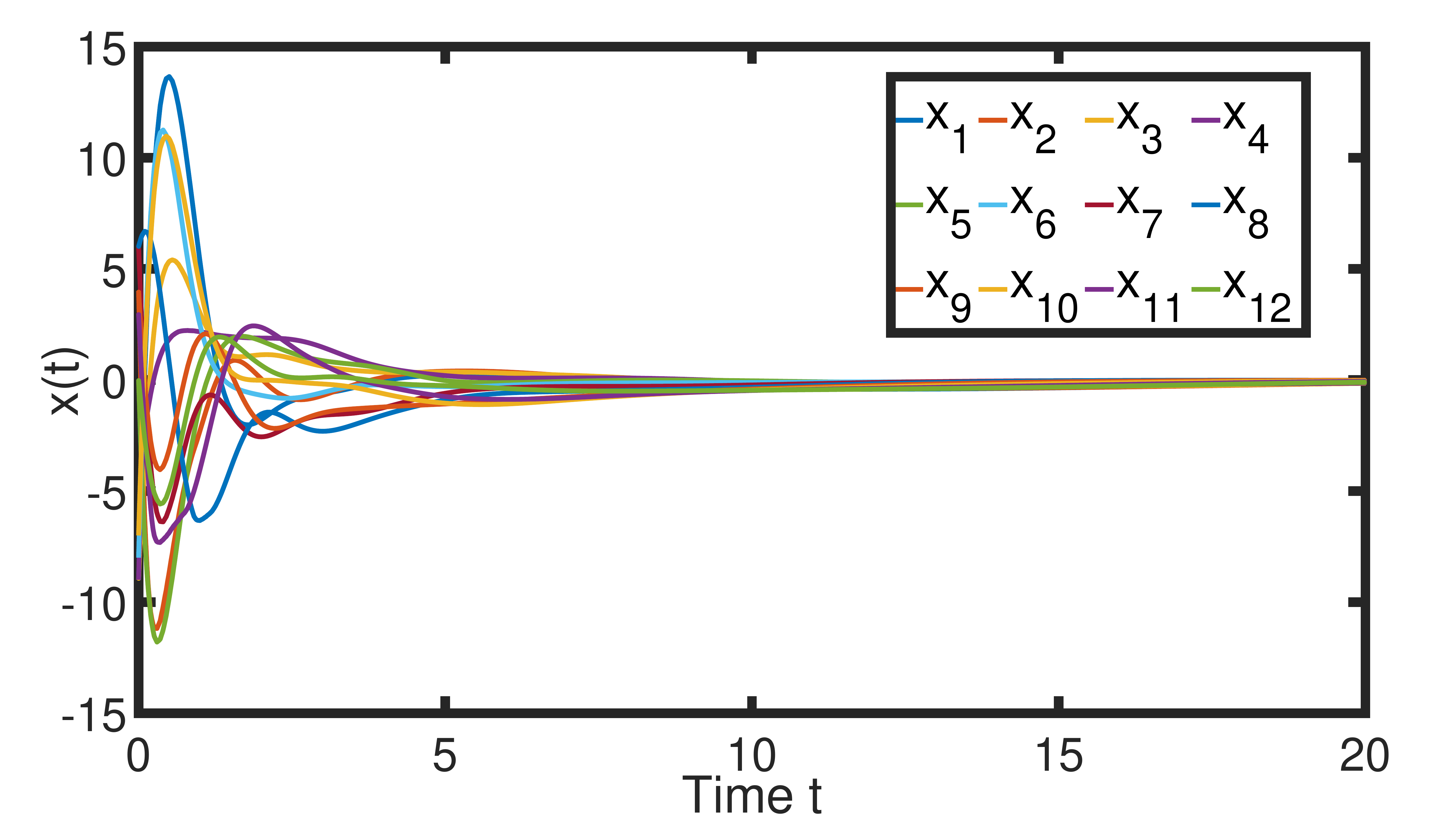}}
	\subfigure{\includegraphics[height=3cm,width=7cm]{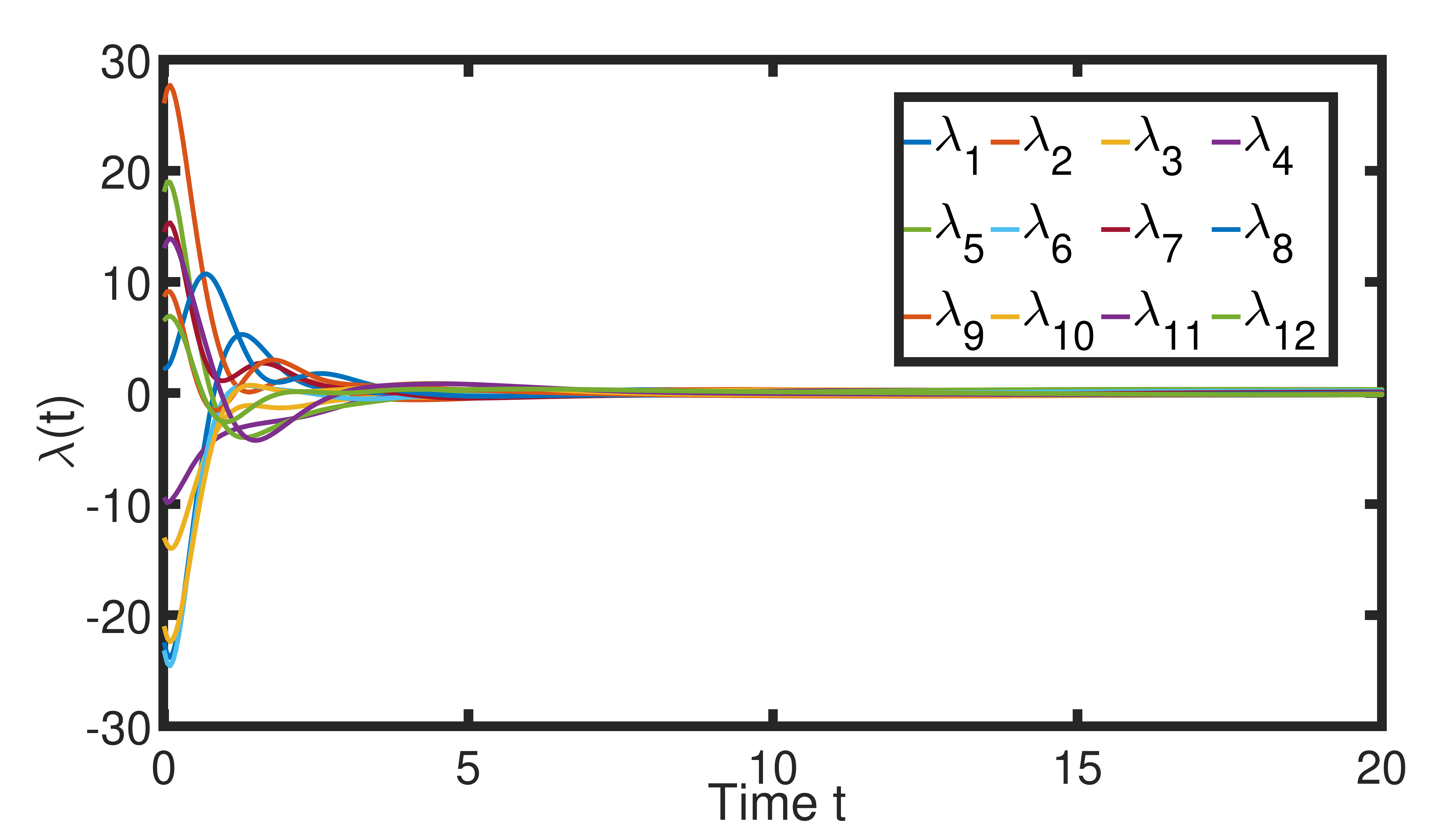}}\label{lambda}
	\vspace{-1.0em}
	\caption{The trajectories of $x_{i}(t)$ and $\lambda_{i}(t)$ for $i\in\{1,\cdots,12\}$.}
\end{figure}

ii) The tracking errors $\|\boldsymbol{x}(t)-\boldsymbol{x}^*\|$ of the QDPD algorithm (blue line) and the related PD algorithm in \cite{liangshu} (red line) are shown in Fig.3. The result illustrates that the quantized communication has a slight impact on the convergence rate. Further, a performance function $J(t)=e^{0.01t}\|\boldsymbol{x}(t)-\boldsymbol{x}^*\|$ is used to show the linear convergence.
\vspace{-2em}
\begin{figure}
	\centering
	\subfigure{\includegraphics[height=3cm,width=7cm]{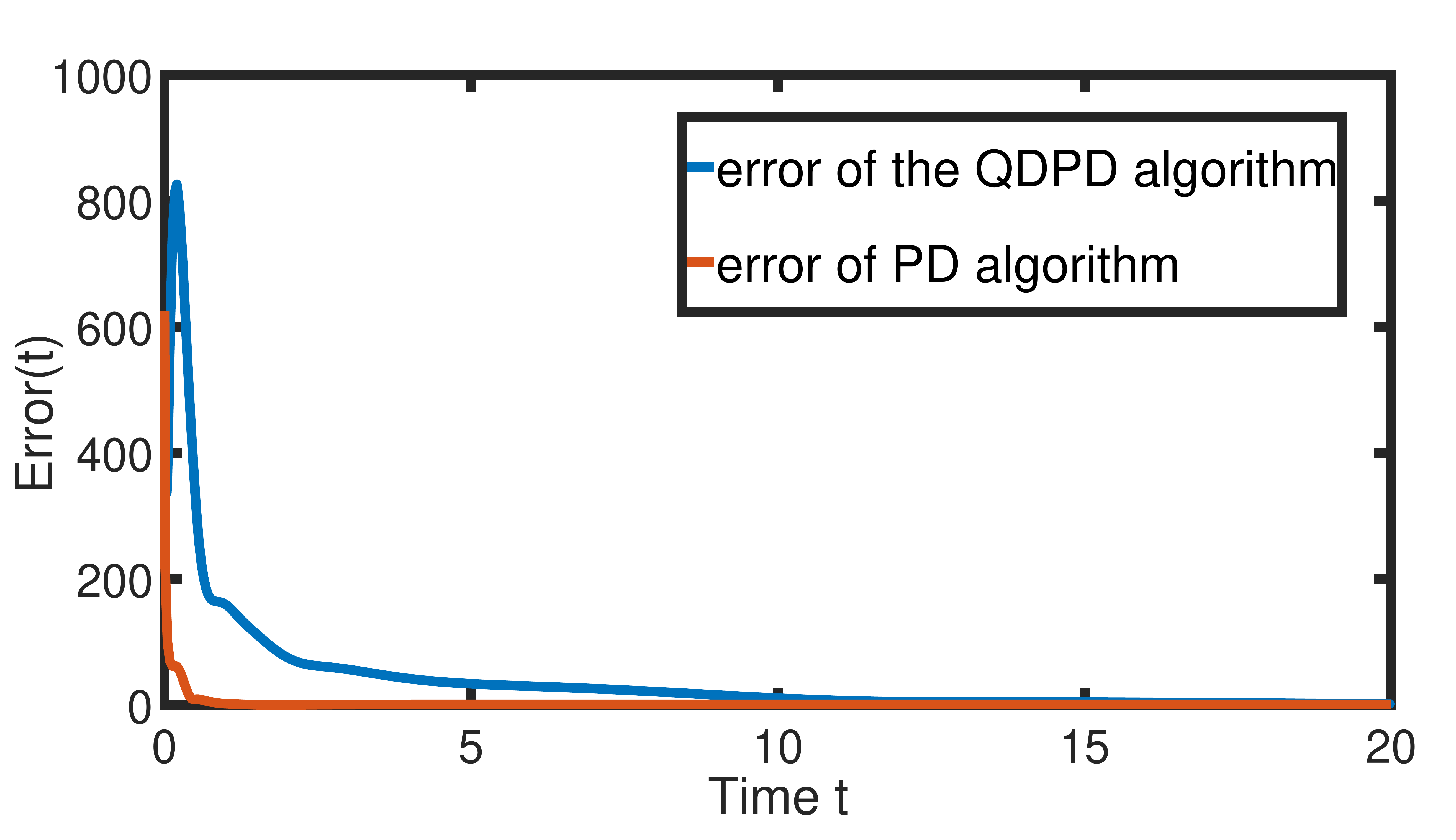}}\label{error}
	\subfigure{\includegraphics[height=3cm,width=7cm]{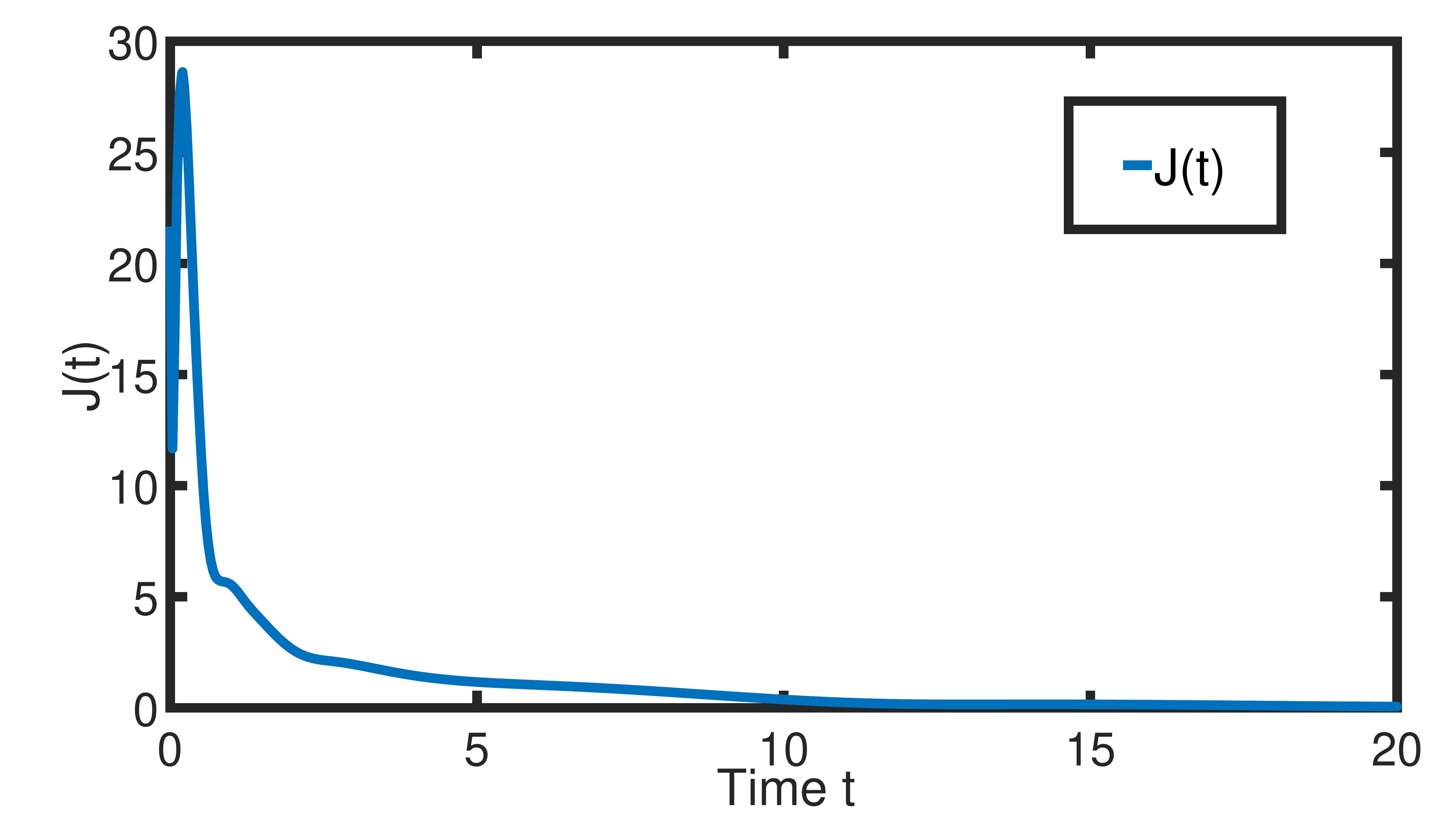}}\label{J}
	\vspace{-1.0em}
	\caption{The trajectories of the tracking errors of the QDPD algorithm and the distributed PD algorithm without quantization, and the trajectory of $J(t)$.}
\end{figure}

\section{conclusion}\label{conclusion}
This paper introduced a continuous-time quantized distributed optimization algorithm, called the QDPD algorithm, to solve a distributed optimization problem with finite bandwidth constraints. Without strong convexity, our QDPD algorithm can converge to an optimal solution at a linear convergence rate under a mild metric subregularity assumption. Meanwhile, the required bandwidth was analyzed in detail. Particularly, for any positive bandwidth, our QDPD algorithm with a minor modification always maintained linear convergence. Under a specific bandwidth, we proved that a lower bound that guarantees this linear convergence rate could be characterized.

\section{Appendix}
To complete the proof of Theorem 1, we first introduce two properties concerned on the upper bound of the stacked optimal dual variables $\|\boldsymbol{\lambda}^*\|$ and the bound of $V(\boldsymbol{z})$.
\begin{prop}
Under Assumption 3, the stacked optimal dual variables $\|\boldsymbol{\lambda}^*\|\!\!\leq\! \frac{1_{Nn}^{T}\boldsymbol{\lambda}(0)}{\sqrt{Nn}}\!+\!\sqrt{Nn}(\frac{M_{1}}{\sigma_{2}}\!+\!M_{2})$.
\end{prop}
\begin{proof}
For simplicity, we discuss the case of $n=1$ since the proof of the case $n>1$ is similar. From \eqref{compactalgorithm}, $1_{N}^{T}\dot{\boldsymbol{\lambda}}=0,$ which leads to $1_{N}^{T}\boldsymbol{\lambda}(t)=1_{N}^{T}\boldsymbol{\lambda}(0)$ for any $t\geq 0$. Using the symmetry of $L_{\mathcal{G}}$, we decompose the Laplacian matrix
$L_{\mathcal{G}}\!\!=\!\!P^{T}\!A\!P$ and its generalized inverse matrix $L_{\mathcal{G}}^{\dagger}\!\!=\!\!P^{T}\!A^{\dagger}\!P$ respectively associated with an orthogonal matrix $P$ and
\begin{eqnarray}
A\!\!=\!\!\text{diag}(0,\sigma_{2},\cdots,\sigma_{N}),~~~A^{\dagger}\!\!=\!\!\text{diag}(0,\frac{1}{\sigma_{2}},\cdots, \frac{1}{\sigma_{N}}).\nonumber
\end{eqnarray}
Multiplying both side of \eqref{kkt1} by $L_{\mathcal{G}}^{\dagger}$ can lead to that
\begin{flalign}
	\left\{\begin{array}{l}-P^{T}[A^{\dagger}P\nabla f(\boldsymbol{x}^*)+\!AA^{\dagger}P(\boldsymbol{x}^*+\boldsymbol{\lambda}^*)]\!=\!0_{N}, \\ 
	P^{T}A^{\dagger}AP\boldsymbol{x}^*\!=\!0_{N}.\end{array}\right.\label{kkt3}
\end{flalign}
Define $\widehat{\boldsymbol{x}}^*=P\boldsymbol{x}^*$, $\widehat{\boldsymbol{\lambda}}^*=P\boldsymbol{\lambda}^*$ and $\nabla \widehat{f}(\boldsymbol{x}^*)=P\nabla f(\boldsymbol{x}^*)$. Then substituting them into \eqref{kkt3} obtains that  
\begin{equation}
\widehat{\lambda}^*_{i}=\widehat{x}^*_{i}-\frac{1}{\sigma_{i}}\nabla \widehat{f}_{i}(\boldsymbol{x}^*),~i=2,\cdots,N.\label{qqqq}
\end{equation} 
Further define $\widehat{\boldsymbol{\lambda}}^{*}_{+}\!=\![\widehat{\lambda}^{*}_{2};\!\cdots\!;\widehat{\lambda}^{*}_{N}]\!\in\!{\mathbb{R}^{N-1}}$. Invoking \eqref{qqqq} yields that 
\begin{flalign}
\|\widehat{\boldsymbol{\lambda}}^{*}_{+}\|\!=\!\sqrt{\sum_{i=2}^{N}(\widehat{x}^*_{i}\!-\!\frac{1}{\sigma_{i}}\nabla \widehat{f}_{i}(\boldsymbol{x}^*))^{2}}\leq\frac{1}{\sigma_{2}}\|\nabla {f}(\boldsymbol{x}^*)\|\!+\!\|{\boldsymbol{x}}^*\|.\nonumber
\end{flalign}
Regrading $L_{{\mathcal{G}}}P^{T}\!\!=\!PA$ and $L_{{\mathcal{G}}}P_{1}^{T}\!\!=\!0_{N}$ associated with $P_{1}\!\!=\!\frac{1_{N}^{T}}{\sqrt{N}}$, by noting that $|\widehat{\lambda}^{*}_{1}|\!\!=\!\frac{1_{N}^{T}\boldsymbol{\lambda}^*}{\sqrt{N}}$, following from Assumption 3, one can obtain that
\begin{flalign}
\|\boldsymbol{\lambda}^*\|=\|\widehat{\boldsymbol{\lambda}}^*\|\!\leq\!|\widehat{\lambda}^{*}_{1}|\!+\!\|\widehat{\boldsymbol{\lambda}}^{*}_{+}\|\leq \frac{1_{N}^{T}\boldsymbol{\lambda}(0)}{\sqrt{N}}\!+\!\sqrt{N}(\frac{1}{\sigma_{2}}M_{1}\!+\!M_{2}).\nonumber
\end{flalign}
When $n\neq 1$, $\|\boldsymbol{\lambda}^*\|\leq \frac{1}{\sqrt{Nn}}1_{N}^{T}\boldsymbol{\lambda}(0)\!\!+\!\sqrt{Nn}(\frac{1}{\sigma_{2}}M_{1}\!\!+\!M_{2})$.
\end{proof}

\begin{prop}
Under Assumption 1, the bound of $V(\boldsymbol{z})$ satisfies
\begin{flalign}
\frac{3\sigma_{N}}{2}\|\boldsymbol{z}-\boldsymbol{z}^*\|^2\!\leq\! V(\boldsymbol{z})\leq\frac{m_{f}\!+\!6\sigma_{N}}{2}\|\boldsymbol{z}-\boldsymbol{z}^*\|^2.\label{boundz}
\end{flalign}
\end{prop}
\begin{proof}
By \eqref{v2u}, one has that $V(\boldsymbol{z})\geq\frac{3\sigma_{N}}{2}\|\boldsymbol{z}-\boldsymbol{z}^*\|^2$. Next, we establish the upper bound of $V(\boldsymbol{z})$. It follows from the $m_{f}$-Lipschitz continuity of $\nabla f(\boldsymbol{x})$ that $f(\boldsymbol{x})-f(\boldsymbol{x}^*)-(\boldsymbol{x}-\boldsymbol{x}^*)^{T}\nabla f(\boldsymbol{x}^*)\leq \frac{m_{f}}{2}\|\boldsymbol{x}-\boldsymbol{x}^*\|^2$.  Since $\boldsymbol{L}_{\mathcal{G}}$ is positive semi-defined, $\frac{1}{2}(\boldsymbol{x}-\boldsymbol{x}^*)^{T}\boldsymbol{L}_{\mathcal{G}}(\boldsymbol{x}-\boldsymbol{x}^*)\leq\frac{\sigma_{N}}{2}\|\boldsymbol{x}-\boldsymbol{x}^*\|^2$ and $(\boldsymbol{x}-\boldsymbol{x}^*)^T\boldsymbol{L}_{\mathcal{G}}(\boldsymbol{\lambda}-\boldsymbol{\lambda}^*)\leq \frac{\sigma_{N}}{2}(\varepsilon\|\boldsymbol{x}-\boldsymbol{x}^*\|^2+\frac{1}{\varepsilon}\|\boldsymbol{\lambda}-\boldsymbol{\lambda}^*\|^2)$ for any $\varepsilon>0$. Choosing $\varepsilon\!=\!\frac{\sigma_{N}}{m_{f}+\sigma_{N}}$ results in that
	\begin{flalign}
		V_{2}(\boldsymbol{z})\!\!&\leq\!\! (m_{f}\!\!+\!\sigma_{N})/2\|\boldsymbol{z}\!-\!\boldsymbol{z}^*\|^2\!+\!\sigma_{N}^2/(2(m_{f}\!\!+\!\sigma_{N}))\|\boldsymbol{x}\!-\!\boldsymbol{x}^*\|^2,\nonumber\\
		\!\!&\leq\!\!(m_{f}\!\!+\!2\sigma_{N})/2\|\boldsymbol{z}\!-\!\boldsymbol{z}^*\|^2.
	\end{flalign}
which implies $V(\boldsymbol{z})\leq\frac{m_{f}+6\sigma_{N}}{2}\|\boldsymbol{z}\!-\!\boldsymbol{z}^*\|^2$. Hence, \eqref{boundz} holds.
\end{proof}	

Now, we use the mathematical induction method to prove \eqref{lem22}. When $k=0$, Assumption 3 naturally ensures \eqref{lem22}. We now prove that if \eqref{lem22} holds when $k=k_{1}$ for any $k_{1}\!\!\in\!{\mathbb{N}}$, then \eqref{lem22} holds when $k=k_{1}+1$ by considering the following four steps.

\textbf{Step 1:} We prove that for any $t'\in\left(k_{1}T,(k_{1}+1)T\right]$,
\begin{equation}
	\|\boldsymbol{e}(t)\|\!<\! b(t),~t\!\in\!\left(k_{1}T,t'\right)\!\!\Rightarrow\! V(\boldsymbol{z})\!\leq\! a(t),~t\!\in\!\left(k_{1}T,t'\right].\label{step1}
\end{equation}
For any $\beta\in(0,1)$, \eqref{VVV} is rewritten as 
\begin{flalign}
	\dot{V}\!\!\leq\!\!-\!\frac{\beta}{2}\|\!\boldsymbol{F}(\boldsymbol{z})\|^{2}\!\!-\!\!\frac{1\!\!-\!\beta}{2}\|\!\boldsymbol{F}(\boldsymbol{z})\!\|^{2}\!\!+\!\!(2\sigma_{N}^3\!\!+\!\!\frac{11}{2}\sigma_{N}^2)\|\boldsymbol{e}(t)\|^{2}\!\!+\!\!\frac{4V}{3}.\nonumber
\end{flalign}
By inequalities \eqref{definition1} and \eqref{boundz}, one has that
\begin{eqnarray}
	\dot{V}\!\!\leq\!\! -\!\eta V\!\!-\!\Big{[}\frac{1-\beta}{\kappa^2(m_{f}\!+\!6\sigma_{N})}\!\!-\!\frac{4}{3}\Big{]}V\!\!+\!\!(2\sigma_{N}^3\!\!+\!\!\frac{11}{2}\sigma_{N}^2)\|\boldsymbol{e}(t)\|^2.\nonumber
\end{eqnarray}
 For any $t\!\in\!(k_{1}T,t')$, assume $V\!\geq\!a(t)$. Define $V_{1}(t)\!\!=\!Ve^{\eta t}$, then $V_{1}(t)\!\geq\!\! \frac{m_{f}+6\sigma_{N}}{3\sigma_{N}}V(0)$. Combining with $\|\boldsymbol{e}(t)\|\!<\!b(t)$, we obtain $\dot{V}\!\leq\! -\eta V$, which further leads to 
\begin{flalign}
	\dot{V}_{1}(t)\!\!&=\!\!e^{\eta t}(\dot{V}+\eta V)\leq 0.\label{doV1}
\end{flalign} 
Based on the above analysis, $\{z|V(\boldsymbol{z})e^{\eta t}\!\leq\!\frac{m_{f}+6\sigma_{N}}{3\sigma_{N}}V(0)\}$is a positively invariant set. Recalling $V\leq a(k_{1}T)$  at time $k_{1}T$, it follows $V_{1}(k_{1}T)\!\!\leq\!\frac{m_{f}+6\sigma_{N}}{3\sigma_{N} }V(0)$. Hence, $V_{1}(t)\!\leq\!\! \frac{m_{f}+6\sigma_{N}}{3\sigma_{N}}V(0),~t\!\in\!(k_{1}T,t']$. It implies $V\!\leq\! a(t),~t\!\in\!(k_{1}T,t']$.

\textbf{Step 2:} We prove that for any $t' \in[k_{1}T,(k_{1}+1)T)$,
\begin{equation}
V(\boldsymbol{z})\!\leq\! a(t),~t\!\in\!\left(k_{1}T,t'\right]\!\Rightarrow\!\|\boldsymbol{e}(t)\|\!<\!b(t),~t\!\in\!\left(k_{1}T,t'\right].\label{step2}
\end{equation}

From \eqref{compactalgorithm}, it is easy to obtain that
	\begin{eqnarray}
		\dot{\boldsymbol{e}}\!=\!\left[ \begin{matrix}
			\!-\!\nabla f(\boldsymbol{x})\!\!+\![\boldsymbol{L}_{\mathcal{G}}  ~\boldsymbol{0}_{Nn}]\boldsymbol{e}\!\!+\![\boldsymbol{0}_{Nn}~\boldsymbol{L}_{\mathcal{G}} ]\boldsymbol{e}\!\!-\!\boldsymbol{L}_{\mathcal{G}} \boldsymbol{x}\!\!-\!\boldsymbol{L}_{\mathcal{G}} \boldsymbol{\lambda}\\
			\boldsymbol{L}_{\mathcal{G}} \boldsymbol{x}\!-\![\boldsymbol{L}_{\mathcal{G}} ~\boldsymbol{0}_{Nn}]\boldsymbol{e}
		\end{matrix} \right].\label{dite}
	\end{eqnarray} 
    Define $\widetilde{\boldsymbol{z}}\!\!=\!\boldsymbol{z}\!\!-\!\boldsymbol{z}^*$ and $\Theta\!\!=\!\nabla f(\boldsymbol{x})\!\!-\!\nabla f(\boldsymbol{x}^*)$. Inserting \eqref{kkt1} into \eqref{dite} yields that
	\begin{eqnarray}
		\dot{\boldsymbol{e}}\!-\!\overline{\boldsymbol{L}}_{\mathcal{G}} \boldsymbol{e}\!=\!\left[\begin{matrix}
			-\boldsymbol{L}_{\mathcal{G}}\widetilde{\boldsymbol{x}}\!-\!\boldsymbol{L}_{\mathcal{G}}\widetilde{\boldsymbol{\lambda}}\!-\!\Theta\\
			\boldsymbol{L}_{\mathcal{G}}\widetilde{\boldsymbol{x}}
		\end{matrix} \right],~\overline{\boldsymbol{L}}_{\mathcal{G}}\!=\!\left[\begin{matrix}
		\boldsymbol{L}_{\mathcal{G}} & \boldsymbol{L}_{\mathcal{G}}\\
		-\boldsymbol{L}_{\mathcal{G}} & \boldsymbol{0}
	\end{matrix} \right], \label{dote2}
	\end{eqnarray}
Integrating both side of \eqref{dote2} from $k_{1}T$ to $t,~t\!\in\!(k_{1}T,t']$ leads to that
	\begin{flalign}
		\boldsymbol{e}(t)&=\boldsymbol{e}(k_{1}T)e^{\int_{k_{1}T}^{t}\overline{\boldsymbol{L}}_{\mathcal{G}}d\tau}\!\!-\!e^{\int_{k_{1}T}^{t}\overline{\boldsymbol{L}}_{\mathcal{G}}d\tau}\int_{k_{1}T}^{t}\left[\begin{matrix}
			\boldsymbol{L}_{\mathcal{G}}\widetilde{\boldsymbol{x}}\!+\!\boldsymbol{L}_{\mathcal{G}}\widetilde{\boldsymbol{\lambda}}\!+\!\Theta\\
			-\boldsymbol{L}_{\mathcal{G}}\widetilde{\boldsymbol{x}}
		\end{matrix} \right]\nonumber\\
		\!\!&~~~e^{\!-\int_{k_{1}T}^{\tau}\overline{\boldsymbol{L}}_{\mathcal{G}}d\tau}d\tau.\label{ett}
	\end{flalign}
Then, taking the norm on both sides of \eqref{ett} gives that
		\begin{flalign}
		\|\boldsymbol{e}(t)\|\!\!&\leq\!\!\|\boldsymbol{e}(k_{1}T)\|\|e^{\overline{\boldsymbol{L}}_{\mathcal{G}}(t\!-\!k_{1}T)}\|\!\!+\!\int_{k_{1}T}^{t}\!\Big{(}\!2\|\boldsymbol{L}_{\mathcal{G}}\widetilde{\boldsymbol{x}}\|\!+\!\|\boldsymbol{L}_{\mathcal{G}}\widetilde{\boldsymbol{\lambda}}\|\!+\!\|\Theta\|\!\!\Big{)}\nonumber\\
		\!\!&~~~\|e^{\overline{\boldsymbol{L}}_{\mathcal{G}}(t-\tau)}\|d\tau.\label{e3}
	\end{flalign}
Next, we analyze the upper bound of each item of \eqref{e3} as
	\begin{flalign}
		\|\boldsymbol{e}(k_{1}T)\|\|e^{\overline{\boldsymbol{L}}_{\mathcal{G}}(t-k_{1}T)}\|\!\!=\!\|\sum_{p=0}^{\infty}[\overline{\boldsymbol{L}}_{\mathcal{G}}(t-k_{1}T)]^{p}\|\|\boldsymbol{e}(k_{1}T)\|\nonumber\\
		\!\leq\!\sum_{p=0}^{\infty}\frac{\|[\overline{\boldsymbol{L}}_{\mathcal{G}}(t-k_{1}T)]^{p}\|}{p!}\|\boldsymbol{e}(k_{1}T)\|\!\leq\!\!e^{\overline{\sigma}_{N}T}\|\boldsymbol{e}(k_{1}T)\|.\nonumber
	\end{flalign}
	Following from $V\leq a(t),~t\in(k_{1}T,t']$ and the conclusion of Proposition 2, one can obtain that
	\begin{flalign}
		&\!\!2\int_{k_{1}T}^{t}\|e^{\overline{\boldsymbol{L}}_{\mathcal{G}}(t-\tau)}\|\|\boldsymbol{L}_{\mathcal{G}}\widetilde{\boldsymbol{x}}\|d\tau
		\!\!\leq\!2\sigma_{N}\!\!\int_{k_{1}T}^{t}\sum_{p=0}^{\infty}\frac{|\overline{\sigma}_{N}(t-\tau)|^{p}}{p!}\|\widetilde{\boldsymbol{x}}\|d\tau\nonumber\\
		&\!\!\leq\!2\sigma_{N}e^{\overline{\sigma}_{N}t}\int_{k_{1}T}^{t}e^{-\overline{\sigma}_{N}\tau}d\tau\int_{k_{1}T}^{(k_{1}+1)T}\sqrt{\frac{2V(\boldsymbol{z})}{\sigma_{N}}}d\tau\nonumber\\
		&\!\!\leq\!\frac{4M_{0}}{\eta\overline{\sigma}_{N}}\sqrt{(m_{f}+6\sigma_{N})\sigma_{N}}(e^{\overline{\sigma}_{N}T}-1)(e^{\frac{\eta}{2}T}-1)e^{-\frac{\eta}{2}(k_{1}+1)T}.\nonumber
	\end{flalign}
	Similar arguments can lead to
	\begin{flalign}
		&\int_{k_{1}T}^{t}\|e^{\overline{\boldsymbol{L}}_{\mathcal{G}}(t-\tau)}\|\|\boldsymbol{L}_{\mathcal{G}}\widetilde{\boldsymbol{\lambda}}\|d\tau\nonumber\\
		&~\!\!\leq\!\frac{2M_{0}}{\eta\overline{\sigma}_{N}}\sqrt{(m_{f}+6\sigma_{N})\sigma_{N}}(e^{\overline{\sigma}_{N}T}\!\!-\!1)(e^{\frac{\eta}{2}T}-1)e^{-\frac{\eta}{2}(k_{1}+1)T}.\nonumber
	\end{flalign}
	The last term of \eqref{e3} can be rewritten as
	\begin{flalign}
		&\int_{k_{1}T}^{t}\|e^{\overline{\boldsymbol{L}}_{\mathcal{G}}(t-\tau)}\|\|\Theta\|d\tau\nonumber\!\leq\!\int_{k_{1}T}^{t}\frac{\sum_{p=0}^{\infty}|\overline{\sigma}_{N}(t-\tau)|^{p}}{p!}\|\Theta\|d\tau,\nonumber\\
		&\!\!\leq\!\frac{1}{\overline{\sigma}_{N}}(e^{\overline{\sigma}_{N}T}-1)\int_{k_{1}T}^{(k_{1}+1)T}m_{f}\|\widetilde{\boldsymbol{x}}\|d\tau\nonumber\\
		&\!\!\leq\!\frac{2M_{0}m_{f}}{\eta\overline{\sigma}_{N}}\sqrt{\frac{m_{f}\!+\!6\sigma_{N}}{\sigma_{N}}}(e^{\overline{\sigma}_{N}T}-1)(e^{\frac{\eta}{2}T}-1)e^{-\frac{\eta}{2}(k_{1}+1)T}.\nonumber
	\end{flalign}
	To proceed, denote by $\Delta_{e}^{k_{1}}=\sup_{k_{1}T<t\leq t'}\|\boldsymbol{e}(t)\|$. Substituting all the above inequalities into \eqref{e3} yields that
	\begin{flalign}
		\Delta_{e}^{k_{1}}\!\!&=\!e^{\overline{\sigma}_{N}T}\|\boldsymbol{e}(k_{1}T)\|\!\!+\!\Big(\frac{6M_{0}}{\eta\overline{\sigma}_{N}}\sqrt{(m_{f}\!\!+\!6\sigma_{N})\sigma_{N}}\!\!+\!\frac{2M_{0}m_{f}}{\eta\overline{\sigma}_{N}}\nonumber\\
		&~~\sqrt{\frac{m_{f}\!\!+\!6\sigma_{N}}{\sigma_{N}}}\Big)(e^{\overline{\sigma}_{N}T}-1)(e^{\frac{\eta}{2}T}-1)e^{-\frac{\eta}{2}(k_{1}+1)T}.\nonumber
	\end{flalign}
Choosing $T$ satisfies \eqref{T}, which implies that
	\begin{flalign}
		&\Big(\frac{6M_{0}}{\eta\overline{\sigma}_{N}}\sqrt{(m_{f}\!\!+\!6\sigma_{N})\sigma_{N}}\!\!+\!\frac{2M_{0}m_{f}}{\eta\overline{\sigma}_{N}}\sqrt{\frac{m_{f}\!\!+\!6\sigma_{N}}{\sigma_{N}}}\Big)\nonumber\\
		&~~(e^{\overline{\sigma}_{N}T}-1)(e^{\frac{\eta}{2}T}-1)e^{-\frac{\eta}{2}(k_{1}+1)T}\leq c_{1}b(k_{1}T).
	\end{flalign}
By noting that the eigenvalue of $\overline{\boldsymbol{L}}_{\mathcal{G}}$ is $\overline{\sigma}_{N}=\sqrt{\frac{3+\sqrt{5}}{2}}\sigma_{N}$ and $e^{\overline{\sigma}_{N}T}\|\boldsymbol{e}(k_{1}T)\|\leq c_{2}b(k_{1}T)$ via \eqref{lem22}, following from $c_{1}+c_{2}<1$, we have $\|\boldsymbol{e}(t)\|\!<\!b(k_{1}T)=b(t),~t\!\in\!(k_{1}T,t'].$ 
	
\textbf{Step 3:} We prove that for any
\begin{flalign}
	\begin{split}
&	\|\boldsymbol{e}(t)\|\!\leq\! b(t),~t\in\left[k_{1}T,(k_{1}+1)T\right),\\
&	\!\Rightarrow\!  \|\boldsymbol{e}((k_{1}+1)T)\|\leq c_{2}e^{-\frac{(1+\sqrt{5})\sigma_{N}T}{2}}b((k_{1}\!+\!1)T).\label{step3}
	\end{split}&
\end{flalign}
Denote the left limit of $\boldsymbol{e}(t)$ as $\boldsymbol{e}^{-}(t)$. For any $t\in[k_{1}T,k_{1}+1)$, there is $\boldsymbol{q}^{\boldsymbol{z}}(t)\!=\!\boldsymbol{q}^{\boldsymbol{z}}(k_{1}T)$ such that
	\begin{flalign}
		&~~\boldsymbol{z}((k_{1}+1)T)-\boldsymbol{q}^{\boldsymbol{z}}(k_{1}T)\nonumber\\
		&=\lim_{t\rightarrow(k_{1}+1)T^{-}}\boldsymbol{z}(t)-\boldsymbol{q}^{\boldsymbol{z}}(k_{1}T)=\boldsymbol{e}^{-}((k_{1}+1)T).\nonumber
	\end{flalign}
	 For any $t\!\in\![k_{1}T,k+1)$, one has that $\|\boldsymbol{e}(t)\|\!<\!b(t)\!=\!b(k_{1}T)$. Then, invoking \eqref{L} yields that
	\begin{flalign}
		\left\|\frac{\boldsymbol{z}((k_{1}+1)T)-\boldsymbol{q}^{\boldsymbol{z}}(k_{1}T)}{\l(k_{1}+1)}\right\|&\leq\left\|\frac{\boldsymbol{e}^{-}((k_{1}+1)T)}{l(k_{1}+1)}\right\|\nonumber\\
		&=\Big{\|}\frac{b(k_{1}T)}{\l(k_{1}+1)}\Big{\|}\leq \frac{L}{2},
	\end{flalign}
that is, the quantizer is unsaturated at time $t\!\!=\!(k_{1}\!+\!1)T$. It further ensures the quantization error satisfying
	\begin{eqnarray}
\|\boldsymbol{e}((k_{1}\!+\!1)T)\|\!\!\leq\!\sqrt{\frac{Nn}{2}}\l(k_{1}\!+\!1)\!=\!c_{2}e^{-\sqrt{\frac{3\!+\!\sqrt{5}}{2}}\sigma_{N}T}b((k_{1}\!+\!1)T).\nonumber
	\end{eqnarray}

\textbf{Step 4:} 
At last, we prove \eqref{lem22} holds for $k=k_{1}+1$ based on \eqref{step1}, \eqref{step2} and \eqref{step3}.

Firstly, we need to prove
\begin{eqnarray}
	\mathop{\text{sup}}\limits_{t}\{k_{1}T<t < (k_{1}+1)T\big{|}\|\boldsymbol{e}(t)\|<b(t)\}=(k_{1}+1)T.\nonumber	
\end{eqnarray}

Recalling the definition of $b(t)$, there is $b(t)=b(k_{1}T)$ for any $t\in\left[k_{1}T,(k_{1}+1)T\right)$. Define the set as follows,
\begin{eqnarray}
	\Omega=\{k_{1}T< t<(k_{1}+1)T\big{|}\|\boldsymbol{e}(t)\|<b(k_{1}T)\}.\label{omega}	
\end{eqnarray}
At time $k_{1}T$, the quantization error satisfies $\|\boldsymbol{e}(k_{1}T)\|< b(k_{1}T)$. Since $\boldsymbol{e}(t)$ is continuous on $t\in[k_{1}T,(k_{1}+1)T)$, there exists $\sup_{t\in{\Omega}}t>k_{1}T$ satisfying $\|\boldsymbol{e}(t)\|< b(k_{1}T)$. Next, we use a contradiction argument to prove $\sup_{t\in{\Omega}}t=(k_{1}+1)T$. Assume that $t_{0}=\sup_{t\in{\Omega}}t,~t_{0}\in(k_{1}T,(k_{1}+1)T)$.

Since $\boldsymbol{e}(t)$ is continuous on $t\in[k_{1}T,(k_{1}+1)T)$, there must be $\|\boldsymbol{e}(t_{0})\|=b(k_{1}T)$ by \eqref{omega}. However,
\begin{flalign}
	\begin{split}
	&~~\|\boldsymbol{e}(t)\|< b(k_{1}T)=b(t),~t\in\left(k_{1}T,t_{0}\right), \nonumber\\
		&\stackrel{\eqref{step1}}{\Rightarrow} V\leq a(t),t\in\left(k_{1}T,t_{0}\right],\nonumber\\
		&\stackrel{\eqref{step2}}{\Rightarrow} \|\boldsymbol{e}(t)\|<b(t)=b(k_{1}T),~t\in\left(k_{1}T,t_{0}\right],
	\end{split}&
\end{flalign}
which contradicts to $\|\boldsymbol{e}(t_{0})\|\!=\!b(k_{1}T)$, then $\sup_{t\in{\Omega}}t\!=\!(k_{1}\!+\!1)T$.

Up to now, we have proved that $\|\boldsymbol{e}(t)\|\!<\!b(t),~t\!\in\!(k_{1}T,(k_{1}+1)T)$. Following \eqref{step1} yields that
$V(t)\leq a(t),~t\in[k_{1}T,(k_{1}+1)T]$. Following \eqref{step3} yields that $\|e((k_{1}+1)T)\|\leq c_{2}e^{-\sqrt{\frac{3+\sqrt{5}}{2}}\sigma_{N}T}b((k_{1}+1)T)$. Thus, \eqref{lem22} holds for $k=k_{1}+1$.  

\bibliographystyle{ieeetr}  
\bibliography{nonconvexquantization}

\end{document}